\documentclass[a4paper,11pt]{amsart} 
\usepackage{MiscStyle}
\usepackage{NumberArtStyle}
\usepackage{ShortcutStyle}
\usepackage{tikz-cd}
\usepackage{verbatim}
\usepackage[applemac]{inputenc}    

\def\Res{\operatorname{Res}}
\def\Bl{\operatorname{Bl}}

\begin{document}

\author{J. J.~Nu\~no-Ballesteros, G.~Pe\~nafort Sanchis}

\title[Iterated multiple points I]{Iterated multiple points I: functoriality, defining equations and pathologies}

\address{Departament de Matem\`atiques,
Universitat de Val\`encia, Campus de Burjassot, 46100 Burjassot
SPAIN}
\email{Juan.Nuno@uv.es}
\address{BCAM, Basque Center for Applied Mathematics, Mazarredo 14, E48009 Bilbao, Spain}
\email{gpenafort@bcamath.org}

\thanks{The first author has been partially supported by DGICYT Grant MTM2015--64013--P. 
The second author was partially supported by the ERCEA 615655 NMST Consolidator Grant and 
by the Basque Government through the BERC 2014-2017 program and by Spanish Ministry of Economy and
Competitiveness MINECO: BCAM Severo Ochoa excellence accreditation SEV-2013-0323.}



	\begin{abstract}
This is the first of a two-part work on Kleiman's iterated multiple point spaces. We show  general properties of these spaces, leading to explicit equations describing them for maps (of any corank) between complex manifolds. We also describe pathologies regarding dimension and lack of symmetry. 	\end{abstract}

\maketitle

\section*{Introduction}

 The multiple point spaces of maps $f\colon X\to Y$  are a key tool in many areas, such as enumerative geometry \cite{Kleiman1982PlaneForms,Kleiman1990Multiple-point-,Rimany2002Multiple-Point-,KleimanMultiplePointFormulasI,  Ran1985,Marangell2010The-General-Qua}, the study of Thom polynomials \cite{Kazarian2003Multising,Ohmoto:2014}, the study of the vanishing (co)homology of disentanglements \cite{GoryunovMond, HoustonTop, Mond:2016, PenafortZach2018} and the study of finite determinacy of map-germs \cite{MararMondCorank1,Altintas:2014,AltintasMond2013, Penafort-Sanchis2014THE-GEOMETRY-OF,Penafort-Sanchis:2016a,Marar2012Double-point-cu,Nuno-Ballesteros2015On-multiple-poi}. Despite their relevance, the multiple point spaces are not well understood objects. While it is clear that the multiple point spaces must contain the strict multiple points (i.e., $r$-tuples $(x^1,\dots,x^r)$ such that $f(x^i)=f(x^j)$ and $x^i\ne x^j$, for all $i\ne j$), there is no consensus about the best way to include the diagonals in order to get a reasonable structure.

There are several approaches to the definition of multiple point spaces; some based on deformations \cite{Nuno-Ballesteros2015On-multiple-poi}, on the Hilbert scheme \cite{Kleiman1990Multiple-point-} or on the Fitting  ideals \cite{MondPellikaanFittingIdeals}. Here we study a general approach, Kleiman's \emph{iterated multiple point spaces} \cite{KleimanMultiplePointFormulasI}, defined for any separated morphism of schemes $f\colon X\to Y$. 

The \emph{double point space} of $f$ is 
\[
K_2=\Res_{\Delta X}(X\times_Y X),
\]
the residual space of the fibered product $X\times_Y X$ along the diagonal $\Delta X$. Composition of the structure map with the first projection gives a map $K_2\to X$. Higher order \emph{multiple point spaces} are defined iteratively: the triple point space $K_3$ is the double point space of $K_2\to X$, and comes with a map $K_3\to K_2$. The quadruple point space $K_4$ is the double point space of $K_3\to K_2$, and so on.

If $X$ and $Y$ are smooth and $f$ has only corank one singularities (i.e. if $f$ is curvilinear in Kleiman's terminology), then $K_r$ coincides with Mond's multiple point space $D^r$, the subspace of $X^r$ given by the vanishing of the iterated divided differences (see \cite{MondSomeRemarks}). Since the number of equations is $(r-1)p$, one deduces that $K_r$ is a local complete intersection in $X^r$, whenever it has the correct dimension $rn-(r-1)p$, with $n=\dim X$ and $p=\dim Y$. A remarkable theorem of Marar and Mond \cite{MararMondCorank1} states that a corank one map is stable if and only if all, $K_r$ are smooth of the correct dimension and that a corank one map germ is finitely determined if and only if all $K_r$ are isolated complete intersection singularities of the correct dimension.

The main difficulty with Kleiman's construction is to find explicit equations for $K_r$ in the presence of singularities of corank $\ge 2$. In this paper, we propose an alternative description of $K_r$ which solves this problem for maps between smooth spaces $X$ and $Y$, allowing singularities of any corank. For any fixed $X$, the spaces $K_r$ of the maps $f\colon X\to Y$ can be embedded as closed subspaces of a \emph{universal multiple point space} $B_r=B_r(X)$. By definition, $B_r$ is the multiple point space of the constant map $X\to *$. If $X$ is smooth of dimension $n$, then $B_r$ is smooth of dimension $rn$; indeed, it is the blowup of $B_{r-1}\times_{B_{r-2}} B_{r-1}$ along $\Delta B_{r-1}$. The main result, Theorem \ref{thmKrEqualsMr}, claims that
\[K_r=b^{-1}(K_{r-1}\times_{K_{r-2}} K_{r-1}) : E,\]
where $b\colon  B_r\to B_{r-1}\times_{B_{r-2}} B_{r-1}$ is the blowup map, $E$ is the exceptional divisor and $Z:W$ stands for the zero locus of the quotient  $I_Z:I_W$ of the defining ideal sheaves of two subspaces $Z$ and $W$. From this result, we derive explicit local equations of $K_r$ inside $B_r$, which are natural generalizations of the iterated divided differences in a convenient atlas of $B_r$. Again, $K_r$ is locally defined by $(r-1)p$ equations in $B_r$, hence it is a local complete intersection whenever it has the correct dimension. We remark that for $r=2$, our description of $K_r$ coincides with Ronga's double point space \cite{Ronga1972La-classe-duale}.

The proof of Theorem \ref{thmKrEqualsMr} is based on some nice functorial properties satisfied by $K_r$, found in Section \ref{secFunctoriality}. These properties are shown by introducing the \emph{multiple points functors}  $ {\bf K_r}$ from the arrow category $A'(\cC)$, where the maps $f\colon X\to Y$ are the objects and the 
morphisms from $f$ to $f'$ are diagrams of the form
\[
\begin{tikzcd}
X \arrow["f"]{r} \arrow[hookrightarrow]{d}& Y\arrow{d}\\
X'\arrow["f'"]{r}						& Y' 
\end{tikzcd}
\]
The key functorial property is that the multiple point functors commute with fibered products in $A'(\cC)$, that is,
\[{\bf K_r}(f_1\times_F f_2)={\bf K_r}(f_1)\times_{{\bf K_r}(F)}{\bf K_r}(f_2).\]
This rather abstract property implies more intuitive results, such as the good behavior of $K_r$ under unfoldings, under restrictions, and that multiple points can be computed coordinate-wise  (see Propositions \ref{propKrCoordinatewise} and \ref{propKrUnfolding}).

The final Section \ref{secPathologies} is devoted to pathologies exhibited by $K_r$ when $f$ is a finite map with  singularities of corank $\geq 2$ and $r$ is big enough. The first one is that $K_r$ never has the correct dimension, even when $f$ is stable. In particular, it must have components of different dimensions, since it is dimensionally correct along the corank one points. Second, the image of $K_r$ by $f$ does not coincide with the multiple point space in the target given by Fitting ideals, as defined by Mond and Pellikaan in \cite{MondPellikaanFittingIdeals}. The last pathology is the lack of symmetry of $K_r$ with respect to permutations of coordinates. Even for triple points, the natural action of $S_3$ on $X^3$ cannot be lifted to $K_3$. 

Appendix \ref{appendixIntersections} contains some considerations --relevant to Section \ref{secFunctoriality}-- about intersections in the setup of categories which have fibered products. For better exposition, some technical proofs are given in a separate Appendix \ref{appendixProofs}.

In a forthcoming paper \cite{SecondIterated}, the local properties of $K_r$ and their relation to stability and finite determinacy of maps will be studied. On one hand, we will show that $K_3$ is smooth when $f$ is stable and that it provides a desingularization of the multiple point space $D^3$, which is always singular when $f$ has singularities of corank $\ge 2$. The analogous result for $K_2$ was proven by Ronga in \cite{Ronga1972La-classe-duale}. On the other hand, smoothness fails from quadruple points onwards  for generically one-to-one maps of corank $\geq 2$. These pathologies will be used to give a simple criterion for finite determinacy within a wide range of dimensions, analogous to the Marar-Mond criterion for the corank one case \cite{MararMondCorank1}.

Based on Sections \ref{secEqsForKr} and \ref{secCoordinatesForBr}, we have implemented a library in {\sc Singular} \cite{SingularSoftware} to compute the generalised divided diferences of any polynomial map $f\colon\C^n\to\C^p$, that is, local equations of $K_r$ in the charts of the smooth space $B_r$. The library {\tt IteratedMultPoint.lib} is freely available (see \cite{Library}) and its usage is illustrated in Example \ref{exDoblesTrifoldCone}. We think that the library will be a useful tool for anyone interested in working with examples.

The technical obstacles of higher corank have forced authors to restrict their work to  singularities of corank one. However, the attention paid to higher corank singularities has been growing over the years. They are finding a place in works about finite determinacy \cite{Altintas:2014,MararNunoANoteOnFiniteDeterminacyForCorank2,Marar2012Double-point-cu,Penafort-Sanchis2014THE-GEOMETRY-OF,Penafort-Sanchis:2016a,Mond:2016},  enumerative geometry  \cite{Feher2012Thom-series-of-}, as well as some other topics \cite{Fernandez-de-Bobadilla2006A-reformulation}. We hope that this work will clarify some aspects of the higher corank case. 

 									\section*{Aknowledgements}
	We are grateful to Steve Kleiman, for very helpful discussions during a stay of the second author in Boston.

									\section{Preliminaries}\label{secPreliminaries}
									
	\subsection*{Terminology}
We use preimages and intersections as defined for any category that has fibered products (such as the arrow categories from Section \ref{secFunctoriality}). This is done by replacing the notion of subset by that of monomorphism. Details can be found in Appendix \ref{appendixIntersections}. 

An arrow $X\hookrightarrow Y$ stands for an embedding of complex spaces and, unless otherwise stated, such a map is assumed to be closed.

 In diagrams, the symbol ``$*$'' stands for the complex manifold consisting of a single point.

Given two closed complex subspaces $X=V(I)$ and $Y=V(J)$ of a complex space $Z$, we write $X:Y=V(I:J)\subseteq Z$.

Given a holomorphic map $f\colon X\to Y$ between manifolds, the \emph{corank} of  $f$ at  $x\in X$ is
\[\corank f_x=\dimc(\ker df_x).\]

This work is written in the category of complex spaces, but it can be adapted to the category of schemes with separated morphisms. Also, some results cited here are stated for schemes in the original sources.

	\subsection*{Residual schemes}
Residual spaces are similar to blowups, with the difference that the symmetric algebra plays the role that the Rees algebra plays for blowups. Here we introduce some properties satisfied by residual spaces, and their relation to blowups. Proofs which we could not find in the literature are contained in Appendix \ref{appendixProofs}.
\begin{definition}\label{defResidualSpace}
Given a closed complex subspace $W=V(I)\subseteq X$, the \emph{residual space} of $X$ along $W$ is the relative homogeneous spectrum
\[
\Res_WX=\Proj_X S(I),\]
where $S(I)$ stands for the symmetric algebra of $I$, regarded as an $\cO_X$-module.
The residual space comes equipped with a canonical proper morphism 
\[\Res_WX\stackrel{r}{\to} X.\]
\end{definition}

\begin{prop}\label{propResidualOfSubspaceAlongIntersection}
Given a embedding $X\hookrightarrow \cX$ (not necessarily closed) and a closed embedding $\cW\hookrightarrow \cX$, there is a unique embedding $\Res_{\cW\cap X}X\hookrightarrow \Res_{\cW}\cX$, making the following diagram commutative:
\[
\begin{tikzcd}[column sep=1em]
\Res_{\cW\cap X}X	\arrow[d]\arrow[r,hookrightarrow] &\Res_\cW\cX\arrow[d] \\
X\arrow[r,hookrightarrow]& \cX
\end{tikzcd}
\]
This construction is functorial, in the sense that the monomorphism associated to the composition $X\hookrightarrow X'\hookrightarrow \cX$ is the composition of the ones associated to $X\hookrightarrow X'$ and $X'\hookrightarrow \cX$.
\begin{proof}See Proof \ref{proofResidualOfSubspaceAlongIntersection} in Appendix \ref{appendixProofs}.
\end{proof}
\end{prop}

\begin{prop}\label{propResAndIntersection}
Given two embeddings $X_i\hookrightarrow \cX$ and a closed embedding $\cW\hookrightarrow \cX$, let $W_i=\cW\cap X_i$. Inside $\Res_\cW\cX$, we have 
\[\Res_{W_1\cap W_2}(X_1\cap X_2)=\Res_{W_1}X_1\cap \Res_{W_2}X_2.\]
\begin{proof}See Proof \ref{proofResAndIntersection} in Appendix \ref{appendixProofs}.
\end{proof}
\end{prop}

Let $b\colon \Bl_YX\to X$ be the blowup of $X$ along $Y$. The main property relating residual schemes and blowups is due to Micali \cite{Micali} (see also \cite[Proposition 2.3.1]{KleimanMultiplePointFormulasI}.
\begin{prop}\label{propMicali} Let $Y\subseteq X'\subseteq X$ be closed complex subspaces, with $Y$ regularly embedded in $X$. Then $\Res_{Y}X=\Bl_YX$ and the respective defining ideals satisfy $I_{\Res_YX'}I_{b^{-1}(Y)}=I_{b^{-1}(X')}.$
	\end{prop}

	\begin{rem}\label{remResSubspaceRegEmb}
Taking into account that $b^{-1}(Y)$ is a Cartier divisor,  the previous property can be restated as follows: if $Y$ is regularly embedded in $X$, then	
\[\Res_Y X'=b^{-1}(X'):b^{-1}(Y).\] 
	\end{rem}

	\begin{lem}\label{LemaKleiman}
Let $X\hookrightarrow \cX$ and let $\cW$ be a closed regularly embedded subspace of $\cX$. If $\cW\cap X$ is regularly embedded in $X$, then 
\[
\Res_{\cW\cap X} X=b_{\cW}^{-1}(X):b_\cW^{-1}(\cW)=b_{\cW\cap X}^{-1}(Z):b_{\cW\cap X}^{-1}(\cW\cap X),
\]
for the blowup maps $b_\cW\colon \Bl_\cW \cX\to \cX$ and $b_{\cW\cap X}\colon \Bl_{\cW\cap X}\cX\to \cX$.
\begin{proof}See Proof \ref{proofLemaKleiman} in Appendix \ref{appendixProofs}.
\end{proof}
	\end{lem}

									\section{The multiple point spaces $K_r$}\label{secDefMultPoints}
			\begin{definition} Given a map $f\colon X\to Y$ between complex spaces, the \emph{iterated multiple point spaces} $K_r$ are defined as follows: Let $K_0=Y$ and $K_1=X$, so that $f$ is a morphism $K_1\to K_0$. For $r\ge 2$, we define
\[K_{r}=\Res_{\Delta K_{r-1}}\left(K_{r-1}\times_{K_{r-2}}K_{r-1}\right)\]
and let $K_r\to K_{r-1}$ be the map obtained by composition of  $K_r\to K_{r-1}\times_{K_{r-2}}K_{r-1}$ 
with the first projection $K_{r-1}\times_{K_{r-2}}K_{r-1}\to K_{r-1}$.
In order to distinguish multiple point spaces of different maps, we sometimes write $K_r(f)=K_r$ .
	\end{definition}

By construction, the multiple point spaces satisfy the \emph{iteration principle:}

\begin{prop}\label{iteration}
For any map between complex spaces, and any integers $r,r'\geq 0$ and $s,s'\geq 1$, such that $r+s=r'+s'$, the multiple points satisfy
	\[K_r\left(K_s\to K_{s-1}\right)=K_{r'}\left(K_{s'}\to K_{s'-1}\right).\]
\end{prop}

Away from the preimage of the diagonal, the double point space $K_2$ is isomorphic to $X\times_Y X$, and measures the failure of injectivity of $f$. What goes on over the diagonal is a more delicate matter, specially if we allow $X$ to be singular. 

\begin{prop}
A morphism $X\to Y$ is an embedding if and only if $K_2=\emptyset$. In this case, all higher $K_r$ are also empty.
\begin{proof} A morphism $X\to Y$ is an embedding if and only if the embedding $\Delta X\hookrightarrow X\times_Y X$ is an isomorphism. In turn, this is equivalent to $\Res_{\Delta X}X\times_Y X=\emptyset$. For $r> 2$,  the space $K_r$ is the double point space of $\emptyset=K_{r-1} \to K_{r-2}$. This is an embedding, hence $K_r=\emptyset$.
\end{proof}
\end{prop}

Computing double points directly from the definition is hard and, as a rule, each step of the iteration process makes the problem harder. What follows is an exception where each step is acomplished by blowing up a manifold along a closed submanifold. 

	\begin{prop}\label{propMultPointsOfSubmersion}
If $f\colon X\to Y$ is a submersion between complex manifolds, then all $K_r$ are smooth, all $K_{r}\to K_{r-1}$ are submersions, and  
\[K_{r}=\Bl_{\Delta K_{r-1}}\left(K_{r-1}\times_{K_{r-2}}K_{r-1}\right).\]
%
\begin{proof} By the iteration principle, it suffices to show the result for $r=2$. Since $f$ is a submersion, $\Delta X\subseteq X\times_Y X\subseteq X\times X$ are submanifolds and, by Proposition \ref{propMicali}, $
K_2=\Res_{\Delta X}\left(X\times_Y X\right)=\Bl_{\Delta X}\left(X\times_Y X\right)$.
One checks easily that $K_2\to X$ is submersive.
\end{proof}
	\end{prop}
	The reader may be surprised that we care about multiple points of submersions. After all, multiple points have always been regarded as a tool for the study of finite and generically one-to-one maps. As we will see, submersions do play a role in the study of multiple points of general maps between manifolds. Before being able to compute $K_r$ for possibly non-submersive maps, we need to figure out certain relations between  spaces $K_r=K_r(f)$ and $K'_r=K_r(f')$, in terms of relations satisfied by maps $f\colon X\to Y$ and $f'\colon X'\to Y'$. This is the starting point:

	\begin{prop}\label{propFunctorialityK_k}
\label{propFunctorialityK_kItem1} Every commutative diagram of the form
\[
\begin{tikzcd}
X \arrow{r} \arrow[hookrightarrow]{d}& Y\arrow{d}\\
X'\arrow{r}						& Y' 
\end{tikzcd}
\tag{$D$}\label{D}
\]
 extends uniquely to a  sequence of commutative diagrams 
\[
\begin{tikzcd}
\cdots \arrow{r} &K_3 \arrow{r} \arrow[hookrightarrow]{d}&K_2 \arrow{r} \arrow[hookrightarrow]{d}&X \arrow{r} \arrow[hookrightarrow]{d}& Y\arrow{d}\\
\cdots \arrow{r} & K_3'\arrow{r}&K_2'\arrow{r}	&X'\arrow{r}& Y' 
\end{tikzcd}
\]
satisfying the following properties:
\begin{enumerate}
\item\label{itemCommutativityKr}
The composition $K_r\hookrightarrow K'_r\to K'_{r-1}\times_{K'_{r-2}} K'_{r-1}$ equals the composition $K_r\to K_{r-1}\times_{K_{r-2}} K_{r-1}\hookrightarrow K'_r\to K'_{r-1}\times_{K'_{r-2}} K'_{r-1}$.

\item \label{itemFunctorialityKr} For any commutative diagram
\[\begin{tikzcd}[row sep=1.5em]
X \arrow{r} \arrow[hookrightarrow]{d}& Y\arrow{d}\\
X' \arrow{r} \arrow[hookrightarrow]{d}& Y'\arrow{d}\\
X''\arrow{r}						& Y''
\end{tikzcd}
\]
the embedding $K_r\hookrightarrow K''_r$ is the composition $K_r\hookrightarrow K'_r\hookrightarrow K''_r$.

\item \label{itemIsomorphicKr}
If $X\to X'$ is an isomorphism and $Y\to Y'$ is an embedding, then $K_r\hookrightarrow K'_r$ is an isomorphism.
\end{enumerate}
\begin{proof} By the iteration principle, it suffices to show the first step of the extension, because every square in the extended diagram
 is of the form (\ref{D}). Notice that the embedding $X\times_YX\hookrightarrow X'\times_{Y'}X,'$ has  $\Delta X$ as the preimage of $\Delta X'$. Moreover, the construction of this embedding is functorial in the same sense of Item \ref{itemFunctorialityKr}, and gives an isomorphism if the hypothesis of Item \ref{itemIsomorphicKr} are met. Applying Proposition \ref{propResidualOfSubspaceAlongIntersection}, we obtain the embedding $K_2\hookrightarrow K'_2$, also in a functorial way. Therefore, the resulting embedding also satisfies  Items \ref{itemFunctorialityKr} and \ref{itemIsomorphicKr}.
  \end{proof}
	\end{prop}

The simplest diagram (\ref{D}) gives an interesting outcome: for a fixed space $X$, the multiple point spaces of all maps $X\to Y$ can be embedded into unique universal spaces.

	\begin{definition}\label{defUniversal}
The \emph{universal $r$-th multiple point space} of a complex space $X$ is 
\[B_r=K_r(X\to *),\] where $X\to *$ is the constant map. 
In order to distinguish the universal spaces of different complex spaces, we write $B_r(X)=B_r$.
	\end{definition}

	\begin{prop}\label{propUnivMultSpaces}The universal multiple point spaces satisfy the following: 
\begin{enumerate}
\item \label{propUnivMultSpacesItem1}For every map $f\colon X\to Y$, there is a canonical  embedding \[K_r\hookrightarrow B_r.\]
\item \label{propUnivMultSpacesItem2}For every embedding $X\hookrightarrow X'$, there is a canonical embedding
\[B_r(X)\hookrightarrow B_r(X').\] 
\end{enumerate}
The construction of $B_r$ is functorial in the sense that the embeddings in Item \ref{propUnivMultSpacesItem2} commute with compositions. The embeddings of Items \ref{propUnivMultSpacesItem1} and \ref{propUnivMultSpacesItem2}  commute with the morphisms $K_r\to K_{r-1}$, $B_r\to B_{r-1}$ and $B'_r\to B'_{r-1}$, and they are compatible with the structure maps to the fibered products. 
\begin{proof}
Extend, respectively, the following two diagrams:
 \[
\begin{tikzcd}
X \arrow{r} \arrow[equal]{d}& Y\arrow{d}\\
X\arrow{r}						& * 
\end{tikzcd}
\qquad\qquad\begin{tikzcd}
X \arrow{r} \arrow[hookrightarrow]{d}& *\arrow{d}\\
X'\arrow{r}						& * 
\end{tikzcd}
\qedhere\]
\end{proof}
	\end{prop}

	\begin{prop}\label{propK_rProjection}For the projection $T\times X\to T$, we have $K_r=T\times B_r(X)$, and $K_r\to K_{r-1}$ is the product of the identity $T\to T$ and $B_r(X)\to B_{r-1}(X)$.
\begin{proof}
Letting $B_r=B_r(X)$, the claim follows inductively from the isomorphisms $(T\times B_{r-1})\times_{(T\times B_{r-2})}(T\times B_{r-1})\cong T\times (B_{r-1}\times_{B_{r-2}}B_{r-1})$ and  $\Res_{T\times \Delta B_{r-1}}(T\times B_{r-1}\times_{B_{r-2}}B_{r-1})\cong T\times\Res_{\Delta B_{r-1}}(B_{r-1}\times_{B_{r-2}}B_{r-1}).$ 
\end{proof}
	\end{prop}

\begin{prop} \label{propBrForManifolds}
The universal multiple point spaces of a complex manifold $X$ are the blowups $B_r=\Bl_{\Delta B_{r-1}}(B_{r-1}\times_{B_{r-2}}B_{r-1}).$
\begin{proof}
This is a particular case of Proposition \ref{propMultPointsOfSubmersion}, because $X\to *$ is a submersion.
\end{proof}
\end{prop}
By abuse of notation, all these blowup maps are written as
\[b\colon B_r\to B_{r-1}\times_{B_{r-2}}B_{r-1}\]
and all their exceptional divisors as $E=b^{-1}(\Delta B_{r-1})$.

\begin{ex}\label{exUniversalDoubleCn}
The universal double point space of $\C^n$ can be described, globally, as the set
\[B_2(\C^n)=\{(x,x',u)\in \C^n\times \C^n\times\P^{n-1}\mid u\wedge(x'-x)=0\}.\]
\end{ex}
Now consider a map $f\colon \C^n\to Y$, with $Y$ a manifold. Once embedded in $B_2(\C^n)$, the double point space $K_2$ of $f$ has a nice geometric interpretation, going back to work of Ronga  \cite{Ronga1972La-classe-duale}.

\begin{prop}\label{propK2ForSmooth}
As a set, the space $K_2$ of a map $f\colon \C^n\to Y$ between manifolds consists of the following points:
\begin{enumerate}
\item Strict points $(x,x',[x'-x])$, with $x\neq x'$ and $f(x)=f(x')$,  
\item Diagonal points $(x,x,u)$, with $u\in \P(\ker df_x)$.
\end{enumerate}
\end{prop}

We finish the section with an observation on the importance of functoriality. In a diagram of the form  (\ref{D}), the spaces $K_r,B_r$ and $K'_r$ are canonically embedded in $B_r'$; but the functoriality of the involved constructions gives us more: extending the horizontal arrows in the diagram
\[
\begin{tikzcd}[sep=0.36em]
	&X \arrow[rr] \arrow[dl,equal] \arrow[dd,hookrightarrow] 	&	&  Y \arrow[dd] \arrow[dl] \\
 X \arrow[rr,crossing over] \arrow[dd,hookrightarrow]&&  *& \\
	&X' \arrow[rr,swap] \arrow[dl,equal] && Y' \arrow[dl] \\
 X' \arrow[rr] && *\arrow[from=uu,crossing over]
\end{tikzcd}
\] 
we obtain that the compositions $K_r\hookrightarrow  B_r\hookrightarrow  B'_r$ and $K_r\hookrightarrow  K'_r\hookrightarrow  B'_r$ are equal. By the universal property of the intersection of $B_r$  and  $K'_r$ in $B'_r$, we obtain the following result:

\begin{prop}For any diagram of the form  (\ref{D}), there is a unique embedding
\[K_r\hookrightarrow K'_r\cap B_r,\]
compatible with the embeddings $K_r\hookrightarrow K_r'$ and $K_r\hookrightarrow B_r$.
\end{prop}
To the effect of exploiting functoriality further, it is convenient to add a further layer of formalism.

				\section{Functoriality}\label{secFunctoriality}

 The \emph{arrow category} $A(\cC)$ of a category $\cC$ has as objects the arrows of $\cC$, and as morphisms $f\to f'$ the commutative diagrams of the form
\[\begin{tikzcd}
X \arrow["f"]{r}\arrow{d}& Y\arrow{d}\\
X'\arrow["g"]{r}						& Y'
\end{tikzcd}
\]
 which are simply called squares. The identity square and the composition of squares are defined in the obvious way. We write $A'(\cC)$ and $A''(\cC)$, respectively, for the wide subcategories of $A(\cC)$ (that is, subcategories including all objects but only some of the morphisms) whose squares are, respectively, of the forms
\[\begin{tikzcd}
X \arrow["f"]{r}\arrow[hookrightarrow]{d}& Y\arrow{d}\\
X'\arrow["f'"]{r}						& Y'
\end{tikzcd}
\quad\text{and}\quad\begin{tikzcd}
X \arrow["f"]{r}\arrow[hookrightarrow]{d}& Y\arrow[hookrightarrow]{d}\\
X'\arrow["f'"]{r}						& Y'
\end{tikzcd}
\]

Now let $\cC$ be the category of complex spaces. Thanks to Proposition \ref{propFunctorialityK_k}, we may disguise Kleiman's  multiple point spaces as functors. 

\begin{definition}
 The \emph{double point functor} is
\[ {\bf K_2}\colon A'(\cC)\to A''(\cC),\]
taking a map $X\to Y$ to the map $K_2\to X$. For morphisms, ${\bf K_2}$ is given by
 \[\begin{tikzcd}
X \arrow[]{r}\arrow[hookrightarrow]{d}& Y\arrow{d}\\
X'\arrow[]{r}						& Y'
\end{tikzcd}
\quad\longmapsto\quad
\begin{tikzcd}
K_2 \arrow[]{r}\arrow[hookrightarrow]{d}& X\arrow[hookrightarrow]{d}\\
K'_2\arrow[]{r}						& X'
\end{tikzcd}\]
In accordance with the iteration principle, the \emph{iterated multiple point functors} are defined as
\[{\bf K_r}={\bf K_2}\circ\overset{r-1}{\cdots}\circ {\bf K_2}.\]
\end{definition}
These functors give ${\bf K_r}(f)\in\Hom(K_r,K_{r-1})$, for any map $f\colon X\to Y$. To be consistent with the convention that $K_1=X$ and $K_0=Y$, the functor $\bf{K_1}$ is set to be the identity $A'(\cC)\to A'(\cC)$.

  If a category $\cC$ has fibered products, then $A(\cC),A'(\cC)$ and $A''(\cC)$ have them as well. Indeed, the fibered product $f_1\times _F f_2$ of two arrows $f_i\colon X_i\to Y_i$, each equipped with a square to $F\colon \cX\to \cY$, is the canonical arrow 
\[X_1\times_\cX X_2\to Y_1\times_\cY Y_2,\] 
equipped with the two squares formed by $X_1\times_\cX X_2\to X_i$ and $Y_1\times_\cY Y_2\to Y_i$, for $i=1,2$. These arrows fit together in a commutative cube diagram
\[
\begin{tikzcd}[column sep={3.5em,between origins}, row sep={2.6 em,between origins}]
&Y_1\times_\cY Y_2 \arrow[rr,swap] \arrow[dd] 	&	&  Y_1  \arrow[dd] \\
 X_1\times_\cX X_2 \arrow[rr,crossing over] \arrow[ru,"f_1\times_Ff_2"]\arrow[dd]&&  X_1\arrow[ur]& \\
	&Y_2 \arrow[rr,swap] && \cY  \\
 X_2 \arrow[rr] \arrow[ur]&& \cX\arrow[from=uu,crossing over]\arrow[ur]
\end{tikzcd}
\]
which we call a \emph{cartesian cube}. The universal property of $f_1\times _F f_2$ gives, for any other morphism $P\to Q$, equipped with squares forming such a commutative cube, a unique factorisation through the cartesian cube (see the first diagram in the proof of Theorem \ref{KrCommutesWithFibProdInA'(C)} below). Note that, since the morphisms in $A''(\cC)$ are the monomorphisms in $A(\cC)$, the fibered products in $A''(\cC)$ are the intersections in $A(\cC)$ (see Definition \ref{defIntersectionObject} for the meaning of intersection in this context).
This sets the ground for the key result of this section.
 
 	\begin{thm}\label{KrCommutesWithFibProdInA'(C)}
The multiple point functors commute with fibered products, that is,
\[{\bf K_r}(f_1\times_F f_2)={\bf K_r}(f_1)\times_{{\bf K_r}(F)}{\bf K_r}(f_2).\]
In particular, for all $r\geq 1$, the multiple point spaces satisfy $K_r(f_1\times_F f_2)=K_r(f_1)\cap K_r(f_2),$ the intersection being made inside $K_r(F)$.
\begin{proof}
The statement is trivial for $r=1$ and, by the iteration principle, it suffices to show the claim for $r=2$. We  show that ${\bf K_2}(f_1\times_Ff_2)$ satisfies the universal property of ${\bf K_r}(f_1)\times_{{\bf K_r}(F)}{\bf K_r}(f_2)$. Observe that ${\bf K_r}(f_1)\times_{{\bf K_r}(F)}{\bf K_r}(f_2)$ is a fibered product in $A''(\cC)$, hence an intersection. 

If the maps are $f_i\colon X_i\to Y_i$, endowed with squares to $F\colon \cX\to \cY$, for simplicity we write
\[X=X_1\cap X_2,\quad K=K_2(f_1\times_F f_2),\quad K_i=K_2(f_i)\quad \text{and}\quad \cK=K_2(F).\]
Following Proposition \ref{propUnivPropIntersectionObject}, it suffices to show that, for any $P\to Q$, commuting with the solid arrows in the diagram
\[
\begin{tikzcd}[column sep=0.7em, row sep=0.8em]
&Q\arrow[ddrrrr,crossing over, bend left =50]\arrow[ddrr,dashed]\arrow[ddddrr,bend right=55,swap]\\
P\arrow[ur]\arrow[ddddrr,bend right=55]\arrow[ddrr,dashed, crossing over]\\
&&&X \arrow[rr,swap,hookrightarrow] \arrow[dd,hookrightarrow] 	&	&  X_1  \arrow[dd,hookrightarrow] \\
&& K \arrow[rr,crossing over,hookrightarrow] \arrow[ru]&&  K_1\arrow[ur]\arrow[from=uullll,crossing over, bend left=50]& \\
&&	&X_2 \arrow[rr,swap,hookrightarrow] && \cX  \\
&& K_2 \arrow[rr,hookrightarrow] \arrow[ur]\arrow[from=uu,crossing over,hookrightarrow]&& \cK\arrow[from=uu,crossing over,hookrightarrow]\arrow[ur]
\end{tikzcd}
\]
there exist the indicated dashed arrows, making the following diagrams commutative:
\[
 \begin{tikzcd}[row sep=1.7em,column sep=1em]
P \arrow[rr]	\arrow[d,dashed]&	& Q\arrow[d,dashed] \\
K\arrow[rr]	&&X
\end{tikzcd}
\quad
 \begin{tikzcd}[row sep=1.7em,column sep=0.45em]
P \arrow[rr]	\arrow[dr,dashed]&	& K_1 \\
	& K\arrow[ur,hookrightarrow]&
\end{tikzcd}
 \quad
 \begin{tikzcd}[row sep=1.7em,column sep=0.45em]
Q \arrow[rr]	\arrow[dr,dashed]&	& X_1 \\
	& X\arrow[ur,hookrightarrow]&
\end{tikzcd}
\]

First of all, observe that the commutativity of the rectangular diagram we have to check follows from the commutativity of the other two triangular diagrams: Since $X\hookrightarrow X_1$ is a monomorphism, it suffices to show the equality of the compositions \[P\dasharrow K\to X\hookrightarrow X_1\quad \text{and}\quad P\to Q\dasharrow X\hookrightarrow X_1,\]
which in turn follows from the commutativity of the triangular diagrams and of the top facet of the cube.

The existence of the commuting map $Q\dasharrow X$ follows from the universal property of the intersection $X=X_1\cap X_2$. We are left with the existence of a morphism $P\dasharrow K$ satisfying the triangular commutativity.

The involved double point spaces are
\[K=\Res_{\Delta X}(X\times_Y X),\quad K_i=\Res_{\Delta X_i}(X_i\times_{Y_i} X_i)\ \ \text{and}\ \  \cK=\Res_{\Delta\cX}(\cX\times_\cY\cX).\]
Moreover, inside $\cX\times_\cY\cX$, we have the equalities
\[\Delta X=\Delta X_1\cap\Delta X_2\quad\text{and}\quad \Delta X_i=(X_i\times_{Y_i} X_i)\cap \Delta\cX.\]
Therefore, we may apply Proposition \ref{propResAndIntersection} to obtain that $K=K_1\times_\cK K_2$, hence the desired commuting map exists by the universal property of the fibered product. This finishes the proof of the equality
${\bf K_r}(f_1\times_F f_2)={\bf K_r}(f_1)\times_{{\bf K_r}(F)}{\bf K_r}(f_2)$.

For the claim that $K_r(f_1\times_F f_2)=K_r(f_1)\cap K_r(f_2),$ just observe that the left vertical arrows of the squares in $A'(\cC)$ are monomorphisms. Therefore, the claim is nothing but the equality of the sources of the arrows ${\bf K_r}(f_1\times_F f_2)$ and ${\bf K_r}(f_1)\times_{{\bf K_r}(F)}{\bf K_r}(f_2).$
\end{proof}
	\end{thm}

\begin{thm}\label{thmKrAsIntersection}For any commutative diagram of the form
\[\begin{tikzcd}
X \arrow{r} \arrow[hookrightarrow]{d}& Y\arrow[hookrightarrow]{d}\\
X'\arrow{r}						& Y' 
\end{tikzcd}
\]
the multiple points satisfy $K_r=K'_r\cap B_r,$ the intersection being made in $B'_r$.
\begin{proof}
Since $Y\hookrightarrow Y'$ induces an isomorphism $X\times _YX\simeq X\times_{Y'}X$, one may as well assume that we are dealing with a commutative diagram of the form 
\[\begin{tikzcd}[row sep=7]
X \arrow{rd} \arrow[hookrightarrow]{dd}\\
& Y\\
X'\arrow{ru}						 
\end{tikzcd}
\]
Now apply Theorem \ref{KrCommutesWithFibProdInA'(C)} to the cartesian cube
\[
\begin{tikzcd}[column sep={2.3em,between origins}, row sep={2em,between origins}]
&Y \arrow[rr,swap] \arrow[dd,equal] 	&	&  *  \arrow[dd] \\
 X \arrow[rr,crossing over,equal] \arrow[ru]&&  X\arrow[ur]& \\
	&Y \arrow[rr,swap] && *  \\
 X' \arrow[rr,equal] \arrow[ur]\arrow[from=uu,crossing over, hookrightarrow]&& X'\arrow[from=uu,crossing over,hookrightarrow]\arrow[ur]
\end{tikzcd}
\qedhere\]
\end{proof}
\end{thm}

\begin{prop}\label{propKrCoordinatewise}
For any complex space $X$, the following properties hold inside $B_r$:
	\begin{enumerate}
\item 
Given two maps $f_i\colon X\to Y_i$, let $f=(f_1,f_2)\colon X\to Y_1\times Y_2$. Then
\[K_r(f)= K_{r}(f_1)\cap K_{r}(f_2).\]
\item
Given a map $f\colon X\to Y$ and two subspaces $X_1,X_2\subseteq X$, 
\[K_r(f\vert_{X_1\cap X_2})= K_r(f\vert_{X_1})\cap K_r(f\vert_{X_2}).\]
	\end{enumerate}
\begin{proof}
Apply Theorem \ref{KrCommutesWithFibProdInA'(C)} to the cartesian cubes
\[\begin{tikzcd}[column sep={2.6em,between origins}, row sep={2em,between origins}]
&Y_1\times Y_2 \arrow[rr,swap] \arrow[dd] 	&	&  Y_1  \arrow[dd] \\
 X \arrow[rr,crossing over,equal] \arrow[ru]&&  X\arrow[ur]& \\
	&Y_2 \arrow[rr,swap] && *  \\
 X \arrow[rr,equal] \arrow[ur]\arrow[from=uu,crossing over, equal]&& X\arrow[from=uu,crossing over,equal]\arrow[ur]
\end{tikzcd}
\quad\quad
\begin{tikzcd}[column sep={2.6em,between origins}, row sep={2em,between origins}]
&Y \arrow[rr,swap,equal] \arrow[dd,equal] 	&	&  Y  \arrow[dd,equal] \\
 X_1\cap X_2 \arrow[rr,crossing over,hookrightarrow] \arrow[ru]&&  X_1\arrow[ur]& \\
	&Y \arrow[rr,swap,equal] && Y  \\
 X_2 \arrow[rr,hookrightarrow] \arrow[ur]\arrow[from=uu,crossing over, hookrightarrow]&& X\arrow[from=uu,crossing over,hookrightarrow]\arrow[ur]
\end{tikzcd}
\qedhere\]
\end{proof}
\end{prop}
The following result of Kleiman (see \cite[Proposition 2.4]{KleimanMultiplePointFormulasI}) can be recovered from the previous results.

\begin{prop}\label{propKrUnfolding}
For any map $F\colon T\times X\to T\times Y$ of the form $F(t,x)=(t,f_t(x))$, the following hold:
\begin{enumerate}
\item $K_r(F)\subseteq T\times B_r(X)$.
\item $K_r(f_{t_0})=K_r(F)\cap\{t=t_0\}$.
\end{enumerate}
\begin{proof}The first item follows by putting  Propositions \ref{propK_rProjection} and \ref{propKrCoordinatewise} together. For the second item, apply Theorem \ref{thmKrAsIntersection} to the diagram
\[\begin{tikzcd}
\{t_0\}\times X \arrow{r} \arrow[hookrightarrow]{d}&\{t_0\}\times Y\arrow[hookrightarrow]{d}\\
T\times X\arrow{r}						&T\times  Y 
\end{tikzcd}
\]
and observe that, once $K_r(F)$ is embedded in $T\times B_r(X)$, the intersection $K_r(F)\cap B_r(\{t_0\}\times X)$ is nothing but $K_r(F)\cap\{t=t_0\}$.
\end{proof}

\end{prop}

Theorem \ref{thmKrAsIntersection} has less obvious applications than the previous ones. As an example, we sketch the computation of double points of reflection maps, introduced in \cite{Penafort-Sanchis:2016a}.

\begin{ex}
 Let $G$ be a  reflection group acting on $\C^n$. The orbit map (or quotient map) of $G$ is a polynomial map $\C^n\stackrel{\omega}{\longrightarrow}\C^n,$ taking a $G$-orbit to a point, that is, $\omega^{-1}(\omega(x))=Gx,$ for all $x\in \C^n$.
 A \emph{reflection map} is a map $f\colon X\to \C^n$ forming a commutative diagram 
\[\begin{tikzcd}[row sep=7]
X \arrow["f"]{rd} \arrow[hookrightarrow,"h"]{dd}\\
& \C^n\\
\C^n\arrow["\omega"]{ru}						 
\end{tikzcd}
\]
for some embedding $h$. 

As it turns out, some basic theory of reflection groups suffices to describe the double point space of the orbit map, which is a union of smooth components indexed by $G$. To be precise, the double points are the reduced space
\[K_2(\omega)=\bigcup_{g\in G\setminus \{1\}}B_g,\]
where each $B_g\subseteq B_2(\C^n)$ is obtained by blowing up the graph of the map $g\colon \C^n\to \C^n$ along $\Delta(Fix\, g)$. 

The double point space $K_2(f)$ can be computed easily from $K_2(\omega)$, by Theorem \ref{thmKrAsIntersection}, and it inherits the same $G$-indexed decomposition. This machinery was used to show the finite determinacy of new and very degenerate families of examples of map-germs $\C^n\to \C^{2n-1}$, for all $n$. For instance, the maps 
\[(x,y)\mapsto(x^a,y^b,(x+y)^p),\]
\[(x,y,z)\mapsto(x^a,y^b,z^c,(x+y +z)^p,(x-y +2z)^q,(x+y-2z)^r),\]
 are $\cA$-finitely determined if the integers $a,b,c,p,q,r$ are pairwise coprime and $p,q,r$ and $f$ are odd.

\end{ex}

\section{A formula for $K_r$ inside $B_r$}\label{secFormulaKrInBr}
We give an explicit expression for $K_r$, embedded in $B_r$, in the case where $X$ is a complex manifold. This will lead us to the formulas for $K_r$ from Section \ref{secEqsForKr}.

 For any $r\geq 2$, we may assume inductively that we have embedded  $K_{r-1}$ in $B_{r-1}$ and that we have a map $K_{r-2}\to B_{r-2}$ (for the initial case of $r=2$, we take the identity map $X\to X$ and $Y\to *$). This induces an embedding \[K_{r-1}\times_{K_{r-2}} K_{r-1}\hookrightarrow B_{r-1}\times_{B_{r-2}}B_{r-1}.\] 
Recall that we write $b\colon B_r\to B_{r-1}\times_{B_{r-2}}B_{r-1}$ for the blowup maps, and $E$ for their exceptional divisors.

\begin{thm}\label{thmKrEqualsMr}For any  map $f\colon X\to Y$  between manifolds and any $r\geq 2$,
\[K_r=b^{-1}(K_{r-1}\times_{K_{r-2}} K_{r-1}): E\]

\begin{proof}We write $M_r=b^{-1}(K_{r-1}\times_{K_{r-2}} K_{r-1})\colon E$ and show $M_r=K_r$ in three steps.

	{\bf Step 1:} We show the property analogous to Theorem \ref{thmKrAsIntersection} for $M_r$: For any  commutative diagram of maps between manifolds   
\[\begin{tikzcd}[row sep=7]
X \arrow{rd} \arrow[hookrightarrow]{dd}\\
& Y\\
X'\arrow{ru}						 
\end{tikzcd}
\]
 we have $M_r=M'_r\cap B_r$. In other words, $M_r=h^{-1}(M'_r)$, for the embedding $B_r\hookrightarrow B'_r$ canonically induced by $X\hookrightarrow X'$. 
 
We proceed by induction on $r$. We set $M_0=M'_0=Y$, $M_1=X$ and $M_1'=X'$, so that the cases $r=0,1$ are trivial. By induction, we may assume that $r\geq 2$ and that the statement holds up to $r-1$. Consider the commutative diagram 
\[
\begin{tikzcd}[column sep={2em}, row sep={2.5em}]
 B_r  \arrow[hookrightarrow,"h"]{r}\arrow["b"]{d}&B'_r\arrow["b'"]{d}\\
 (B_{r-1}/B_{r-2})^2 \arrow[hookrightarrow,"\overline h"]{r}&(B'_{r-1}/B'_{r-2})^2
\end{tikzcd}
\]
 and, for convenience, write $\overline M=(M_{r-1}/M_{r-2})^2$ and $\overline{M'}=(M'_{r-1}/M'_{r-2})^2$. 

We claim that $b^{-1}(\overline M)= h^{-1}(b'^{-1}(\overline {M'}))$. Indeed, by induction we have 
\[\overline {M}= (h^{-1}(M'_{r-1})/h^{-1}(M'_{r-2}))^2=(\overline h)^{-1}(\overline {M'}),\]
and the claim follows from the commutativity $\overline h \circ b=b'\circ h$.

Let $E$ and $E'$ be the exceptional divisors on $B_{r}$ and $B'_r$, and observe that $E=h^{-1}(E')$. We have the two equalities
\[M_r= h^{-1}(b'^{-1}(\overline {M'})):h^{-1}(E'),\]
\[h^{-1}(M'_r)= h^{-1}(b'^{-1}(\overline {M'}):E').\]
Therefore, it suffices to show 
\[h^{-1}(b'^{-1}(\overline {M'})):h^{-1}(E')= h^{-1}(b'^{-1}(\overline {M'}):E').\]
In turn, this equality may be rewriten as
\[(b'^{-1}(\overline {M'})\cap B_r):(E'\cap B_r)=((b'^{-1}(\overline {M'})):E')\cap B_r.\]

Now write $p\colon B'_r\to B'_{r-1}$ for the usual map, $pr_1\colon(B'_{r-1}/B'_{r-1})^2\to B'_{r-1}$ for the first projection and let \[Z=p^{-1}(M'_{r-1}).\] 
From the inclusion $p(b'^{-1}(\overline {M'}))=(pr_1\circ b')(b'^{-1}(\overline {M'}))\subseteq M'_{r-1}$, we obtain $b'^{-1}(\overline {M'})\subseteq Z$. Consequently, both sides of the equality we need to check are subspaces of $Z$, and the equality may be restated as
\[(J+S):(I+S)=(J:I)+S,\]
for the ideal sheaves $J,S$ and $I$ defining $b'^{-1}(\Gamma')$, $B_r\cap Z$ and $E'\cap Z$ in $Z$, respectively.

 The right to left inclusion holds in general, so we are left with the inclusion $(J+S):(I+S)\subseteq(J:I)+S$. Now $\Delta B'_{r-1}\cap pr^{-1}(M'_{r-1})=\Delta M'_{r-1}$ is a complex subspace of $\overline {M'}$, and, applying $b'^{-1}$, it follows that  $E'\cap Z$ is a complex subpace of $b'^{-1}(\overline {M'})$ or, equivalently, that $J\subseteq I$. Moreover, the ideal $I$ is principal, because $E'$ is a the exceptional divisor in $B_r'$, so locally we may write
 $I=\langle \lambda\rangle$.  

Now let $a\in (J+S): (I+ S)=( J+S): \lambda$. We have $\lambda a=j+s,$ with $j\in J$ and $s\in S$. Since $J\subseteq I$, then $s=c\lambda$, for some $c$. Since $B_r$ is smooth and is not contained in the exceptional divisor, we have $S: \lambda= S$, and hence $c \in S$.  We have $\lambda(a-c)=j\in J$, so $a-c\in J:\lambda$, and therefore $a\in (J:\lambda)+S$.

	{\bf Step 2:} We reduce the proof to show $K_r=M_r$ for submersions. For any map $X\to Y$ between manifolds, consider the commutative diagram 
\[\begin{tikzcd}[row sep=12]
X \arrow{rd} \arrow[hookrightarrow,"\Gamma"]{dd}\\
& Y\\
X\times Y\arrow["\pi"]{ru}						 
\end{tikzcd}
\]
where $\Gamma$ is the graph embedding and $\pi$ the projection on the second factor. The embedding $\Gamma$ induces embeddings $B_r(X)\hookrightarrow B_r(X\times Y)$, and $\pi$ is obviously a submersion. Assuming that $K_r=M_r$ for submersions, from Theorem \ref{thmKrAsIntersection} and the previous step we obtain the equality
\[K_r(f)=K_r(\pi)\cap B_r(X)=M_r(\pi)\cap B_r(X)=M_r(f)\]
inside $B_r(X\times Y)$, as desired.

{\bf Step 3:} We show $M_r=K_r$ for submersions. By induction, we may assume  $M_{r-1}=K_{r-1}$ and $M_{r-2}=K_{r-2}$. We know from Proposition \ref{propMultPointsOfSubmersion}  that all maps $K_r\to K_{r-1}$ are submersions between smooth spaces. This in turn implies the smoothness of \[Y=\Delta B_{r-1}, \quad Z=M_{r-1}\times_{M_{r-2}}M_{r-1},\quad\text{and}\quad X=B_{r-1}\times_{B_{r-2}}B_{r-1}.\]  Since submanifolds are regularly embedded,  from Lemma \ref{LemaKleiman} we obtain
\begin{align*}
M_r&=b^{-1}\left(M_{r-1}\times_{M_{r-2}}M_{r-1}\right):b^{-1}(\Delta B_{r-1})\\
	&=\Res_{\Delta M_{r-1}}\left(M_{r-1}\times_{M_{r-2}}M_{r-1}\right)=K_r.&\qedhere
\end{align*}
\end{proof}
\end{thm}

\begin{rem}
Theorem \ref{thmKrEqualsMr} was already known for $r=2$, that is, for double points of maps between smooth spaces, see for example \cite[Section 9.3]{Fulton1998Intersection-Th}. However, unlike other results found here, this one does not admit an easy reduction to the case $r=2$. Observe that the source and target of higher maps $K_r\to K_{r-1}$ may not be smooth, and that we don't know a priori that $M_r=b^{-1}(K_{r-1}\times_{K_{r-2}} K_{r-1})\colon E$ satisfies any short of iteration principle. In general, arguments about the algebraic structure of higher $K_r$ are delicate, and it is apparent that our proof relies heavily on the machinery from previous sections.
\end{rem}

									\section{Coordinates for the universal spaces $B_r$}\label{secCoordinatesForBr}
We give coordinates for $B_r(X)$ in the case where $X$ is a complex manifold. First, we justify that it suffices to give coordinates for the spaces $B_r(\C^n)$

 \begin{definition}
Let $\sP$ be a partition $r_1+\dots+r_s=r$ of $r$. We say that a point $z\in B_r$ is of type $\sP$ if the components of  its projection in $X^r$ consist of $s$ different points $x^{(i)}\in X$, each repeated $r_s$ times (in a possibly disordered way).
\end{definition}

Given a point $z\in B_r$ of type $\sP$, we may take pairwise disjoint open subsets $U_1,\dots U_s\subseteq X$, so that $x^{(i)}\in U_i$. For brevity, we write $f^{(i)}=f\vert_{U_i}$.

\begin{prop}\label{propMrAndPartitions}
Let $f\colon X\to Y$ be a map between complex manifolds. Around a point $z\in B_r$ of type $\sP$, the space $K_r(f)$ is locally isomorphic to 
\[K_{r_1}(f^{(1)})\times_Y\dots\times_Y K_{r_s}(f^{(s)}).\]
\begin{proof}See Proof \ref{proofMrAndPartitions} in Appendix \ref{appendixProofs}.
\end{proof}
\end{prop}

\begin{rem}
Let $z\in B_r$ be a point of type $\sP$ as above, for a manifold $X$. Then $B_r$ is locally isomorphic at $z$ to 
\[B_{r_1}(U_1)\times\dots\times B_{r_s}(U_s),\]
for some disjoint coordinate open subsets $U_\alpha$, which we may regard them as subsets $U_\alpha\subset\C^n$. Since both $U_\alpha$ and $\C^n$ are smooth of the same dimension, it follows from Propositions \ref{propBrForManifolds} and \ref{propUnivMultSpaces} that $B_r(U_\alpha)$ is an open submanifold of $B_r(\C^n)$. Hence, coordinates for $B_r(U_\alpha)$ are obtained by restriction of the coordinates for $B_r(\C^n)$.
\end{rem}
\subsection*{A pyramid of maps}
We need one last ingredient, the triangular diagram below, before giving coordinates for the spaces $B_r$. We write the $s$-fold fibered product of a map $X\to Y$ as
\[(X/Y)^s=X\times_Y\stackrel{s}{\dots}\times_Y X.\] 
Unless otherwise stated, maps $(X/Y)^{s}\to (X/Y)^{s-t}$ are assumed to drop the last $t$ components. To avoid confusion, we write
\[(X/Y)^{s}\stackrel{\beta}{\longrightarrow} X\] for the map which keeps the last component and write $\tau\colon K_{r+1}\to K_{r}$ for the composition $K_{r+1}\stackrel{}{\longrightarrow}K_{r}\times_{K_{r-1}}K_{r}\stackrel{\beta}{\longrightarrow}K_{r}$.

	\begin{lem}\label{lemTriangleK}
For any $X\to Y$, there are unique maps
\[(K_{r+1}/K_{r})^s\to (K_{r}/K_{r-1})^{s+1},\] such that the following two diagrams commute:
\[
\begin{tikzcd}[sep=1em]
	 \cdots\arrow[r]&K_4\arrow[d]\\
	 \cdots\arrow[r]&(K_3/K_2)^2\arrow[r]\arrow[d]&K_3\arrow[d]\\
	 \cdots\arrow[r]&(K_2/X)^3\arrow[r]\arrow[d]&(K_2/X)^2\arrow[r]\arrow[d]&K_2\arrow[d]\\
	 \cdots\arrow[r]&(X/Y)^4\arrow[r]&(X/Y)^3\arrow[r]&(X/Y)^2\arrow[r]&X\\
\end{tikzcd}
\tag{$T_K$}\label{TSubK}\]
where $K_{r+1}\to (K_{r}/K_{r-1})^2$ are the structure maps, and 
\[
\begin{tikzcd}[column sep= 8,row sep=1.8em]
	(K_{r+1}/K_{r})^s\arrow[r]\arrow[d,"\beta"']&	(K_{r}/K_{r-1})^{s+1}\arrow[d,"\beta"]\\
	K_{r+1}\arrow[r,"\tau"]&K_r
\end{tikzcd}
\]
\begin{proof} See Proof \ref{proofTriangleK} in Appendix \ref{appendixProofs}
\end{proof}
	\end{lem}

\begin{rem}\label{remTSubB}
As a particular case of Lemma \ref{lemTriangleK}, for any complex space $X$ there is a unique analogous commutative diagram  
\[
\begin{tikzcd}[sep=1em]
	 \cdots\arrow[r]&B_4\arrow[d]\\
	 \cdots\arrow[r]&(B_3/B_2)^2\arrow[r]\arrow[d]&B_3\arrow[d]\\
	 \cdots\arrow[r]&(B_2/X)^3\arrow[r]\arrow[d]&(B_2/X)^2\arrow[r]\arrow[d]&B_2\arrow[d]\\
	 \cdots\arrow[r]&X^4\arrow[r]&X^3\arrow[r]&X^2\arrow[r]&X\\
\end{tikzcd}
\tag{$T_B$}\label{TSubB}\]
\end{rem}

Now we are ready to give coordinates for $B_r=B_r(\C^n)$. For technical reasons, we need to describe the spaces $(B_ r/B_{ r-1})^s$ in the diagram \ref{TSubB} and the maps between them. To gain intuition, we start with the universal double and triple points for $\C^2$. This will also fix the notation for double and triple points of maps corank $\leq 2$, as in Examples \ref{exGenItDivDiffCor2} and \ref{exDoblesTrifoldCone}. 

\subsection*{Atlas for the universal double point space $\boldsymbol{B_2(\C^2)}$} Let $X=\C^2$. As stated in Example \ref{exUniversalDoubleCn},  the universal double point space of $X$ is 
\[B_2=\{((x,y),(x',y'),(u_1:u_2))\in \C^2\times \C^2\times \P^1\mid (x'-x,y'-y)\wedge (u_1,u_2)=0\}.\]
The space $B_2$ is covered by the open subsets $U_1,U_2$, where
\[U_i=\{\big((x,y),(x',y'),(u_1:u_2)\big)\in B_2\mid u_i\neq 0\}.\]
These open subsets are isomorphic to $\C^4$, respectivelly, via the maps
\[((x,y),(x',y'),(u_1:u_2)])\longmapsto ((x,y),x'-x,\frac{u_2}{u_1}),\]
\[((x,y),(x',y'),(u_1:u_2))\longmapsto ((x,y),y'-y,\frac{u_1}{u_2}).\]
Finally, the inverse maps are 
\[((x,y),\lambda,a)\longmapsto\big((x,y),(x+\lambda,y+\lambda a),(1:a)\big),\]
\[((x,y),\lambda,a)\longmapsto\big((x,y),(x+\lambda a,y+\lambda),(a:1)\big),\]
and the exceptional divisor is given by $\lambda=0$ on each $U_i$.
\subsection*{Compatible atlas for $\boldsymbol{B_3(\C^2), (B_2(\C^2)/\C^2)^2}$ and $\boldsymbol{B_2(\C^2)}$} By Proposition \ref{propBrForManifolds}, $B_3$ is the blowup of $(B_2/\C^2)^2$ along $\Delta B^2$, and $(B_2/\C^2)^2$ can be seen as the space of tuples
\[\big((x,y),(x',y'),(x'',y''),[u],[u']\big)\in (\C^2)^3\times (\P^1)^2,\]
such that
\[(x'-x,y'-y)\wedge u=(x''-x,y''-y)\wedge u'=0.\]
By unicity of the maps in Lemma \ref{lemTriangleK}, the map $(B^2)^2\to (\C^2)^3$ is given by 
\[\big((x,y),(x',y'),(x'',y''),[u],[u']\big)\mapsto ((x,y),(x',y'),(x'',y'')\big),\]
because it satisfies the corresponding commutativites.

Now we cover $(B_2/\C^2)^2$ and $B_2$ by open coordinate subsets in a way that allows us to compute $B_3$, and to express the maps $B^3\to(B^2/\C^2)^2$ and $(B^2/\C^2)^2\to (\C^2)^3$ conveniently. It is easy to see (and it follows from Lemma \ref{lemCoveringCollection}) that, by setting $L_1(x,y)=x$, $L_2(x,y)=y$ and $L_3(x,y)=x+y$, we have a covering of $(B_2/\C^2)^2$ by the three open subsets 
\[U_i^2=\{\big((x,y),(x',y'),(x'',y''),[u],[u']\big)\in (B_2/\C^2)^2 \mid L_i(u)\neq 0\neq L_i(u')\}.\]
It is clear that $(B_2/\C^2)^2\to B_2$ restricts to $U_1^2\to U_1$ and $U_2^2\to U_2$, so we shall add to our covering of $B_2$ a new open subset 
\[U_3= \{\big((x,y),(x',y'),[u]\big)\in B_2\mid L_3(u)\neq 0\},\] to have the corresponding restriction $U_3^2\to U_3$.
The isomorphism $U_3\to \C^4$ and its inverse are given by
\[\big((x,y),(x',y'),[u]\big)\longmapsto \big((x,y),(x'-x+y'-y,\frac{u_1}{u_1+u_2})\big),\]
\[\big((x,y),(\lambda,a)\big)\longmapsto \big((x,y),(x+\lambda a,y+\lambda(a-1) ),(a:a-1)\big).\]

To give coordinates to the new open subsets, fix some other linear forms $L'_i$, each of them linearly independent to the corresponding $L_i$, for example $L'_1=y$, $L'_2=x$ and $L'_3=x$. We have the isomorphisms $U_i^2\to \C^6$ mapping a point $\big((x,y),(x',y'),(x'',y''),[u],[u']\big)$ to the point
\[\Big(\big(x,y\big),\,\big(L_i(x'-x,y'-y),\frac{L'_i(u)}{L_i(u)}\big),\,\big(L_i(x''-x,y''-y),\frac{L'_i(u')}{L(u')}\big)\Big).\]
Our choices for $i=1,2,3$ map the point $\big((x,y),(x',y'),(x'',y''),[u],[u']\big)$, respectively, to the point
\[\big((x,y),(x'-x,\frac{u_2}{u_1}),(x''-x,\frac{u'_2}{u'_1})\big),\]
\[\big((x,y),(y'-y,\frac{u_1}{u_2}),(y''-y,\frac{u'_1}{u'_2})\big),\]
\[\big((x,y),(x'-x+y'-y,\frac{u_1}{u_1+u_2}),(x''-x+y''-y,\frac{u'_1}{u'_1+u'_2})\big).\]
The inverse isomorphisms $\C^6\to U_i^2$ map a point $\big((x,y),(\lambda,a),(\lambda',a')\big)$, respectively to the point
\[\big((x,y),(x+\lambda,y+\lambda a),(x+\lambda',y+\lambda' a'),(1:a),(1:a')\big),\]
\[\big((x,y),(x+\lambda a,y+\lambda),(x+\lambda' a',y+\lambda'),(a:1),(a':1)\big),\]
\[\big((x,y),(x+\lambda a,y+\lambda(a-1) ),(x+\lambda' a',y+\lambda'(a'-1)),(a:a-1),(a':a'-1)\big).\]

To compute $B_3$, observe that on each $U_i^2$ the diagonal $\Delta B^2$ is regularly embedded by the equations
\[L_i(x',y')=L_i(x'',y''),\quad \frac{L'_i(u)}{L_i(u)}=\frac{L'_i(u')}{L_i(u')},\] and  is mapped to the set $\{\lambda=\lambda',a=a'\}$. Consequently, $B_3$ can be described as the set of tuples
\[\big((x,y),(x',y'),(x'',y''),[u],[u'],[v]\big),\]
satisfying the conditions
\begin{enumerate}
\item $\big((x,y),(x',y'),(x'',y''),[u],[u']\big)\in U_i$, for some $i$, with $1\leq i\leq 3$,
\item $\Big(L_i(x''-x',y''-y'),\dfrac{L'_i(u)}{L_i(u)}-\dfrac{L'_i(u')}{L_i(u')}\Big)\wedge v=0$.
\end{enumerate}

We may cover $B_3$ by six open subsets $U_{ij}$, with $1\leq i\leq 3$ and $1\leq j\leq 2$, of the form 
\[U_{ij}=\{\big((x,y),(x',y'),(x'',y''),[u],[u'],[v]\big)\in B_3\mid L_i(u), L_i(u'),L_j(v)\neq 0.\}\]
The previous local isomorphisms allow us to regard $B_3$ on each $U_{ij}$ as the set of tuples
\[\big((x,y),(\lambda,a),(\lambda',a'),[v]\big), \]
satisfying $L_j(v)\neq 0$ and $(\lambda'-\lambda,a'-a)\wedge v=0$. Each $U_{ij}$ is mapped isomorphically to $\C^6$ by means of 
\[\big((x,y),(\lambda,a),(\lambda',a'),[v]\big)\longmapsto \big((x,y),(\lambda,a),L_j(\lambda'-\lambda,a'-a),\frac{L'_j(v)}{L_j(v)}\big).\]
Our choices for $j=1,2$ map a point $\big((x,y),(\lambda,a),(\lambda',a'),[v]\big)$, respectively, to the point
\[\big((x,y),(\lambda,a),\lambda'-\lambda,\frac{v_2}{v_1}\big),\]
\[\big((x,y),(\lambda,a),a'-a,\frac{v_1}{v_2}\big),\]  The respective inverse isomorphisms map a point $\big((x,y),(\lambda,a),(\mu, b)\big)$ to the point
\[ \big((x,y),(\lambda,a),(\lambda+\mu,a+\mu b),(1:b)\big),\]
\[\big((x,y),(\lambda,a),(\lambda+\mu b,a+\mu),(b:1)\big).\]
The exceptional divisor in $B_3$ is the preimage of the diagonal $\Delta B_2$ by $B_3\to (B_2/\C^2)^2$, and is given in $U_{ij}$ by the equation $\mu=0$.

Fixed $i$ and $j$, we can arrange the coordinates $(x,y)$, $(x',y')$, $(x'',y'')$, $(\lambda,a)$, $(\lambda',a')$ and $(\mu,b)$ of the spaces $\C^2,(\C^2)^2,B_2,(\C^2)^3,(B_2/\C^2)^2$ and $B_3$ into a triangle
\[\left.\begin{array}{lll}
(\mu,b) &  &  \\
(\lambda',a') &(\lambda,a)  &  \\
 (x'',y'')&(x',y')	  &  (x,y)
 \end{array}\right.\]
corresponding to the triangular diagram in Lemma \ref{lemTriangleK}. The expression in local coordinates of the maps $B_2\to (\C^2)^2$, $B_3\to (B_2/\C^2)^2$ and $(B^2/\C^2)^2\to (\C^2)^3$ can be read directly from the maps $\C^4\to U_{i}$, $\C^6\to \widetilde U_{i}$ and $\C^6\to U_{ij}$. For example, for $\alpha=(i,j)=(1,1)$, these maps are determined by 
\[x'=x+\lambda,\quad y'=y+\lambda a, \quad x''=x+\lambda',\quad y''=y+\lambda' a'\]
\[\lambda'=\lambda+\mu,\quad a'=a+\mu b.\]

\subsection*{The general case: atlas for $\boldsymbol{(B_r/B_{r-1})^s}$, for $\boldsymbol{X=\C^n}$}
For the remaining of the section, we fix $X=\C^n$ and an integer $\ell\geq 2$.  We give coordinate coverings of the spaces
\[(B_r/B_{r-1})^s,\text{ with }r+s\leq \ell+1.\] These are the spaces  in a triangular diagram of the form (\ref{TSubB}) of Remark \ref{remTSubB}, having $B_\ell$ on the top. Our coverings are such that the maps 
\[(B_{r+1}/B_{r})^s\to (B_{r}/B_{r-1})^{s+1}\quad\text{and}\quad (B_{r}/B_{r-1})^s\to (B_{r}/B_{r-1})^{s-1}\]
 take an open subset to an open subset. We are mostly interested in $B_r=(B_{r}/B_{r-1})^1$, but the spaces $(B_r/B_{r-1})^s$ with $s\geq 2$ are needed for the construction. The reader may convince themselves of this by inspecting the proof of Proposition \ref{prop:atlas}.

	\begin{definition}
We say that a collection of linear forms 
$$\cL=\{L_i\colon \C^{n+1}\to \C \}_{i=1}^m$$ 
 is a \emph{covering collection} for $(\P^n)^k$ if
 \begin{enumerate}
 \item $m\geq kn+1$,
 \item any $n+1$ elements in $\cL$ are linearly independent.
 \end{enumerate}
	\end{definition}
 
  	\begin{lem}\label{lemCoveringCollection}
 If $\cL$ is a covering collection for $(\P^n)^k$, then the subsets 
 $$V_L=\{(u^{(1)},\dots,u^{(k)})\in (\P^{n})^k\mid L(u^{(i)})\neq 0,\text{ for 
all }i\leq k\},$$
with $L\in \cL$, form an open covering of $(\P^{n})^k$.
 \begin{proof}
 Assume that some point $(u^{(1)},\dots,u^{(k)})\in (\P^n)^k$ is not contained in any $V_L$, that is, for every $L\in \cL$, there is some $i\leq k$, such that $L(u^{(i)})=0$. Letting $\cL_i=\{ L\in \cL\mid L(u^{(i)})=0\}$, we have 
that $\cL=\cL_1\cup\dots\cup \cL_k$. Since $\vert \cL\vert=m>kn$, there is some 
$\cL_i$ with at least $n+1$ elements. This is in contradiction with the assumption
that any $n+1$ elements in $\cL$ are linearly independent.
 \end{proof}
  	\end{lem}
	
 For the remaining of the section, we fix some $r\leq \ell$ and a covering collection $\cL$ for $(\P^{n-1})^{\ell-1}$ and choose, for each $L_i\in \cL$, different linear forms $L'_1,\dots,L'_{n-1}\in \cL\setminus\{L\}$.
 
\begin{definition}
We write $\widehat L_i\colon \{u\in\P^{n-1}\mid L_i(u)\neq 0\}\to \C^{n-1}$ for the map
\[\widehat L_i(u)=\left(\frac{L'_1(u)}{L_i(u)},\dots, \frac{L'_{n-1}(u)}{L_i(u)}\right).\]
\end{definition}

Recall how the coverings for the case of  $X=\C^2$ were indexed: If we are only interested in double points, that is, if $\ell=2$, then two open open subsets indexed by $i=1,2$ are enough. If we are interested in triple points, i.e. $\ell=3$, then we covered the spaces $B_2$ and $(B_2/\C^2)^2$ by three open subsets indexed by $i=1,2,3$. The space $B_3$ was then covered by six open subsets $U_{ij}$, with multi-indices $(i,j)$, with $i=1,2,3$ and $j=1,2$.

In general, we shall define a set of multi-indices $\alpha=(\alpha_1,\dots,\alpha_{r-1})$ for our open cover of $(B_r/B_{r-1})^s$ in such a way that, for every $i\leq r$, the family  $\{L_{\alpha_i}\mid \alpha\in \cS_r\}$ is a covering collection for $(\P^{n-1})^{\ell-i-1}$. 

	\begin{definition}
For $r\geq 2$, we let $\cS_{r}^\ell$ be the set of multi-indices $\alpha=(\alpha_1,\dots,\alpha_{r-1})$, where the $\alpha_i$ are integer numbers satisfying
\begin{itemize}
\item[] $1\le\alpha_1\le (\ell-1)(n-1)+1,$
\item[] $1\le\alpha_2\le (\ell-2)(n-1)+1,$
\item[]$\quad\quad\vdots$
\item[]$1\le\alpha_{r-1}\le (\ell-r-1)(n-1)+1.$
\end{itemize}
If $r>2$, we have a map $\cS_{r}^\ell\to\cS_{r-1}^\ell$ given by 
\[(\alpha_1,\dots,\alpha_{r-1})\mapsto (\alpha_1,\dots,\alpha_{r-2}).\] 
We set $\cS_1^\ell=\{0\}$ and let $\cS^{\ell}_2\to\cS^{\ell}_1$ be the constant map.
	\end{definition}
	
From now on, we fix a positive integer $s$, satisfying $r+s\leq \ell+1$.
	\begin{definition}\label{defOpenNbhdUAlpha}
For each $\alpha\in \mathcal S_ r$, we write $U_\alpha^s$ for the subset of $\C^n\times (\P^{n-1})^{( r-1)(s-1+ r/2)}$, consisting of tuples of points
\[\left.\begin{array}{llllll} 
u^{( r-1, r+s-2)}&\cdots  &u^{( r-1, r-1)}  &  &  &  \\
\vdots&&\vdots  &\ddots  &  &   \\
u^{(1, r+s-2)}&\cdots&u^{(1, r-1)}  &\cdots    &u^{(1,1)}  &  \\
x^{( r+s-2)}&\cdots& x^{( r-1)} & \cdots  &x^{(1)}  &x
\end{array}\right.
\tag{$P$}\label{Pyramid}
\]
with $x,x^{(j)}\in \C^n$, $u^{(i,j)}\in\P^{n-1}$ and $L_{\alpha_i}(u^{(i,j)})\neq 0$, subject to the following iteratively defined conditions: First, set 
\[\gamma^{(0,0)}=x,\ \gamma^{(0,1)}=x^{(1)},\dots,\ \gamma^{(0, r+s-2)}=x^{( r+s-2)}.\]
Then, for each $1\leq i\leq  r$, set $\delta^{(i,j)}=\gamma^{(i-1,j)}-\gamma^{(i-1,i-1)}$, impose the condition
\[u^{(i,j)}\wedge \delta^{(i,j)}=0\] and, for $j=i,\dots, r+s-2$, set
\[\lambda^{(i,j)}=L_{\alpha_i}(\delta^{(i,j)})\in \C,\quad a^{(i,j)}=\widehat L_{\alpha_i}(u^{(i,j)})\in \C^{n-1},\quad \gamma^{(i,j)}=(\lambda^{(i,j)},a^{(i,j)}).\] 

	\end{definition}
The unusual placement of coordinates in (\ref{Pyramid}) is designed to match the diagram ($T_B$) of Remark \ref{remTSubB}. We refer to the columns of (\ref{Pyramid}) decreassingly from $r+s-2$ to $0$, so that the $j$th column is the one containing $x^{(j)}$.

	\begin{prop}\label{prop:atlas}
The space $(B_ r/B_{ r-1})^s$ is a glueing of the spaces $U_\alpha^s$, with $\alpha\in \cS^{\ell}_ r$. For each multi-index $\alpha\in \cS^{\ell}_ r$,  the map $(B_ r/B_{ r-1})^s\to (B_ r/B_{ r-1})^{s-1}$ restricts to the map 
\[U_\alpha^s\to U_\alpha^{s-1}\]
which drops the left column of (\ref{Pyramid}). For $r\geq 2$, the map $(B_ r/B_{ r-1})^s\to (B_{ r-1}/B_{ r-2})^{s+1}$ restricts to the map 
\[U_\alpha^s\to U_{(\alpha_1,\dots,\alpha_{ r-2})}^{s}\]
 which drops the top row of of (\ref{Pyramid}) (in the case of $ r=2$, the index $(\alpha_1,\dots,\alpha_{ r-2})$ is meant to be $0\in \cS^{\ell}_1$).
\begin{proof}We proceed by induction on $ r$. Together with the statement, we will need to show the following extra items, for a point in $U^s_\alpha$ whose coordinates form the pyramid (\ref{Pyramid}):
\begin{enumerate}
\item\label{extraitem1} For positive $j\geq  r$, the pyramid (\ref{Pyramid}) has the $j$th column equal to the $(r-1)$th column if and only if $\delta^{( r,j)}=0$. In particular, the diagonal of $(B_ r/B_{ r-1})^2$ intersects $U^2_\alpha$ at $\delta^{(r, r)}=0$. 
\item\label{extraitem2} For positive $j\geq  r$ and $j'\geq  r$, the pyramid (\ref{Pyramid}) has the $j$th column equal to the $j'$th column if and only if $\delta^{(r,j)}=\delta^{( r,j')}$.
\end{enumerate}
The case of $ r=1$ is trivial: For each $s\geq 1$ we consider a single set $U_0^s=\C^n\times\stackrel{s}{\dots}\times \C^n,$
with coordinates $x^{(s-1)},\dots, x^{(1)},x$. The map $(B_ r/B_{ r-1})^s\to (B_ r/B_{ r-1})^{s-1}$ drops $x^{(s-1)}$. The extra items (\ref{extraitem1}) and (\ref{extraitem2}) are obvious, since $\delta^{(1,j)}=x^{(j)}-x$ by definition.

Assume that the statement is true for $ r-1$. By item (\ref{extraitem1}) of the induction hypothesis, on each $U^2_\alpha$, with $\alpha\in \cS^{\ell}_{ r-1}$,  the diagonal of $(B_{ r-1}/B_{ r-2})^2$ is regularly embedded by the equations $\delta^{( r-1, r-1)}=0$. Consequently, over $U^2_\alpha$ the space $B_ r$ is obtained by adding a new  $u^{( r-1, r-1)}$ to the $x^{(j)}$ and $u^{(i,j)}$ of $(B_{ r-1}/B_{ r-2})^2$, and imposing the condition that $u^{( r-1, r-1)}\wedge \delta^{( r-1, r-1)}=0$. The blowup  $B_ r\to (B_{ r-1}/B_{ r-2})^2$ drops $u^{( r-1, r-1)}$ and, by the induction hypothesis,  $(B_{ r-1}/B_{ r-2})^2\to B_{ r-1}$ drops the remaining $u^{(i, r-1)}$ and $x^{( r-1)}$. Hence $B_ r\to B_{ r-1}$ drops the entire left column, and the fibered product $(B_{ r}/B_{ r-1})^s$ is obtained by adding new columns containing the $u^{(i, r+j-2)}$ and $x^{( r+j-2)}$, for $i=1,\dots, r-1$ and $j=1,\dots,s$, each subject to its corresponding condition. The number of linear forms $L_{\alpha_i}$ that the multi-indices $\alpha\in \cS^{\ell}_ r$ yield for each $1\leq i\leq  r-1$ justifies, by Lemma \ref{lemCoveringCollection}, that the open subsets $U_\alpha^s$ cover $(B_ r/B_{ r-1})^s$. It is obvious that the map defined by dropping the top row  satisfies the commutativities stated in Lemma \ref{lemTriangleK} and therefore it is the desired map $(B_ r/B_{ r-1})^s\to (B_{ r-1}/B_{ r-2})^{s+1}$.

 To show items (\ref{extraitem1}) and (\ref{extraitem2}), let $j\geq  r$ and assume that a point satisfies 
\[u^{( r-1,j)}=u^{( r-1, r-1)},\dots, u^{(1,j)}=u^{(1, r-1)}\text{ and }x^{(j)}=x^{( r-1)}.\]
On $U^s_\alpha$, the equality $u^{( r-1,j)}=u^{( r-1, r-1)}$ is equivalent to 
\[\widehat L_{\alpha_{ r-1}}(u^{( r-1,j)})=\widehat L_{\alpha_{ r-1}}(u^{( r-1, r-1)}),\] that is, to $a^{( r-1,j)}=a^{( r-1, r-1)}$.  By the induction hypothesis, applying item (\ref{extraitem2}) on $U^s_{\alpha_1,\dots,\alpha_{ r-2}}$, it follows that the equations 
\[u^{( r-2,j)}=u^{( r-2, r-1)},\dots, u^{(1,j)}=u^{(1, r-1)},x^{(j)}=x^{( r-1)}\]
  are equivalent to 
$\delta^{( r-1,j)}=\delta^{( r-1, r-1)}$. Now observe that the conditions  $u^{( r-1,j)}\wedge \delta^{( r-1,j)}=0=u^{( r-1, r-1)}\wedge \delta^{( r-1,j)}$ and $ L_{\alpha_{ r-1}}(u^{( r-1,j)})\neq 0\neq  L_{\alpha_{ r-1}}(u^{( r-1, r-1)})$ hold on $U_\alpha$. Therefore, provided that $u^{( r-1,j)}=u^{( r-1, r-1)}$, the equality $\delta^{( r-1,j)}=\delta^{( r-1, r-1)}$ is equivalent to 
\[L_{\alpha_{ r-1}}(\delta^{( r-1,j)})=L_{\alpha_{ r-1}}(\delta^{( r-1, r-1)}),\]
that is, to $\lambda^{( r-1,j)}=\lambda^{( r-1, r-1)}$.
Putting everything together, our initial conditions are equivalent to $\delta^{( r,j)}=0$, and item (\ref{extraitem1}) follows. To show item (\ref{extraitem2}), assume 
\[u^{( r-1,j)}=u^{( r-1,j')},\dots, u^{(1,j)}=u^{(1,j')},x^{(j)}=x^{(j')}.\]
On $U_\alpha$, the condition $u^{( r-1,j)}=u^{( r-1,j')}$ is equivalent to $a^{( r-1,j)}=a^{( r-1,j')}$ and, by the induction hypothesis, the rest of conditions are equivalent to $\delta^{( r-1,j)}=\delta^{( r-1,j')}$. As before, provided that $u^{( r-1,j)}=u^{( r-1,j')}$ holds, the condition that $\delta^{( r-1,j)}=\delta^{( r-1,j')}$ is equivalent to $\lambda^{( r-1,j)}=\lambda^{( r-1,j')}$, because $u^{( r-1,j)}\wedge \delta^{( r-1,j)}=0=u^{( r-1,j)}\wedge \delta^{( r-1,j)}$. Therefore, our initial conditions are equivalent to $\delta^{( r,j)}=\delta^{( r,j')}$. This finishes the proof of item (\ref{extraitem2}).
\end{proof}
	\end{prop}

\begin{definition}
For each $L\in \cL$, we write 
$\Lambda_L\colon \C^n\to \C^n$ for the linear isomorphism 
\[x\mapsto(L(x),L'_1(x),\dots,L'_{n-1}(x)).\]
We write $\nu_L\colon \C^n\to \C^n$ for the map taking a point $\gamma=(\lambda,a),$ with $\lambda\in \C$ and $a=(a_1,\dots,a_{n-1})\in \C^{n-1}$, to the point
\[\nu_L(\gamma)=\lambda\cdot \Lambda^{-1}_L(1,a_1,\dots,a_{n-1}).\]
To simplify notation, once a multi-index $\alpha$ is fixed, we write $\nu_i=\nu_{L_{\alpha_i}}$.
\end{definition}
Reviewing the proof of Proposition \ref{prop:atlas}, it becomes clear that, for each open subset $U^s_\alpha$, the
$x,\gamma^{(1,1)},\gamma^{(2,2)},\dots,\gamma^{( r-1, r-1)},\gamma^{( r-1, r)},\dots,\gamma^{( r-1, r+s-2)}$ are subject to no relations. For $(B_3/B_2)^2$, these are marked with the symbol ``$\bullet$'' in the diagram
\[
\begin{tikzcd}[sep=1em]
\bullet\arrow[d]&\bullet\arrow[d]\\
\circ\arrow[d]&\circ\arrow[d]&\bullet\arrow[d]\arrow[l]\arrow[ll,bend left=34]\\
\circ&\circ&\circ&\bullet\arrow[lll,bend left=50]\arrow[ll,bend left=34]\arrow[l]\\
\end{tikzcd}
\]
The remaining $\gamma^{(i,j)}$, marked with ``$\circ$'', are determined by  the ``$\bullet$'' entries, by means of the relations 
\[\gamma^{(i,j)}=\gamma^{(i,i)}+\nu_{i+1}(\gamma^{(i+1,j)}),\] corresponding to the arrows in the diagram. Then, from the pyramid
\[\left.\begin{array}{llllll} 
\gamma^{( r-1, r+s-2)}&\cdots  &\gamma^{( r-1, r-1)}  &  &  &  \\
\vdots&&\vdots  &\ddots  &  &   \\
\gamma^{(1, r+s-2)}&\cdots&\gamma^{(1, r-1)}  &\cdots    &\gamma^{(1,1)}  &  \\
\gamma^{(0, r+s-2)}&\cdots& \gamma^{(0, r-1)} & \cdots  &\gamma^{(0,1)}  &x
\end{array}\right.\]
with $\gamma^{(i,j)}=(\lambda^{(i,j)},a^{(i,j)})$, the $u^{(i,j)}$ and $x^{(j)}$ are recovered by setting
\[u^{(i,j)}=[\Lambda^{-1}_{L_{\alpha_i}}(1,a_1^{(i,j)},\dots,a_{n-1}^{(i,j)})]\in \P^{n-1},\quad x^{(j)}=x+\gamma^{(0,j)}.\]
This shows how to produce charts for the open subsets $U_\alpha$:
\begin{prop}
Each $U^s_\alpha\subseteq (B_ r/B_{ r-1})^s$ is isomorphic to $(\C^n)^{ r+s-1}$, via the map that takes a point $x,x^{(j)},u^{(i,j)}$ to the point with coordinates $x,\gamma^{(1,1)},\gamma^{(2,2)},\dots,\gamma^{( r-1, r-1)},\gamma^{( r-1, r)},\dots,\gamma^{( r-1, r+s-2)}.$
\end{prop}

Since we are mainly interested in the spaces $B_ r=(B_r/B_{r-1})^1$, we may take $\ell=r$. We write $\cS_r=\cS_r^\ell$ and $U_\alpha=U_\alpha^1$, as well as
\[\gamma^{(i)}=\gamma^{(i,i)},\quad\lambda^{(i)}=\lambda^{(i,i)}\quad\text{and}\quad a^{(i)}=a^{(i,i)}.\]
This way $B_r$ is covered by the open subsets $U_\alpha$, with $\alpha\in \cS_ r$, and we write $\varphi_\alpha\colon U_\alpha\to (\C^n)^ r$ for the charts giving the coordinates 
\[x,\gamma^{(1)},\dots,\gamma^{( r-1)}.\]

									\section{Equations for the multiple point spaces $K_r$}\label{secEqsForKr}
Here we give an explicit set of local equations for $K_r(f)$ in the coordinates of the affine open subsets described above.
\begin{definition}\label{defGenItDivDiff}
With the previous notations, for any $\alpha\in\mathcal S_ r$ and any $f\colon \C^n\to \C^p$, we define the \emph{iterated generalised divided differences} as follows: The first \emph{generalised divided difference} is
\[f_{\alpha_1}[x,\gamma]=\dfrac{f(x+\nu_1(\gamma))-f(x)}{\lambda}\]
The $j$th iterated generalised divided difference is
\begin{align*}
&f_{\alpha_1,\dots,\alpha_j}[x,\gamma^{(1)},\dots,\gamma^{(j)}]=\\
&=\dfrac{f_{\alpha_1,\dots,\alpha_{j-1}}[x,\gamma^{(1)}\dots,\gamma^{(j-2)},\gamma^{(j-1,j)}]-f_{\alpha_1,\dots,\alpha_{j-1}}[x,\gamma^{(1)},\dots,\gamma^{(j-1)}]}{\lambda^{(j)}},\end{align*}
where $\gamma^{(j-1,j)}$ stands for the function $\gamma^{(j-1)}+\nu_{j}(\gamma^{(j)})$. We ommit the multi-indices $\alpha$ if there is no risk of confusion.
\end{definition}

	\begin{thm}\label{div-dif}
The space $K_r\cap U_\alpha$ is mapped isomorphically by $\varphi_\alpha$ to the zero locus of the ideal sheaf generated by 
\[f[x,\gamma^{(1)}],\dots ,f[x,\gamma^{(1)},\dots,\gamma^{(r-1)}].\]
	\end{thm}
	
	\begin{proof} We proceed by induction on $r$.  By definition, we have:
\begin{align*}
K_r&=b^{-1}(K_{r-1}\times_{K_{r-2}}K_{r-1}): b^{-1}(\Delta B_{r-1})\\
&=b^{-1}(K_{r-1}\times_{K_{r-2}}K_{r-1}): b^{-1}(\Delta K_{r-1}),
\end{align*}
since $(K_{r-1}\times_{K_{r-2}}K_{r-1})\cap \Delta B_{r-1}=\Delta K_{r-1}$. We can compute this quotient in $p^{-1}(K_{r-1})$ instead of $B_r$. Indeed, there are inclusions 
\[
\Delta K_{r-1}\subseteq K_{r-1}\times_{K_{r-2}}K_{r-1}\subseteq \pi_1^{-1}(K_{r-1}),
\]
where $\pi_1: B_{r-1}\times_{B_{r-2}}B_{r-1}\to  B_{r-1}$ is the projection onto the first factor, and hence
\[
b^{-1}(\Delta K_{r-1})\subseteq b^{-1}(K_{r-1}\times_{K_{r-2}}K_{r-1})\subseteq p^{-1}(K_{r-1}).
\]

The induction hypothesis states that $p^{-1}(K_{r-1})$ is defined in the coordinate system $(U_\alpha,\varphi_\alpha)$ of $B_r$ by the ideal sheaf $ I $ generated by the coordinate functions of
$$f[x,\gamma^{(1)}],\dots ,f[x,\gamma^{(1)},\dots,\gamma^{(r-2)}].
$$
Regarded as subspaces of $p^{-1}(K_{r-1})$, the spaces $b^{-1}(\Delta K_{r-1})$ and $b^{-1}(K_{r-1}\times_{K_{r-2}}K_{r-1})$ are defined in $(U_\alpha,\varphi_\alpha)$ by the ideal sheafs generated by the class of $\lambda^{(r-1)}$ and the classes of the coordinate functions of
\[
f[x,\gamma^{(1)},\dots,\gamma^{(r-2)}+\nu_{r-1}(\gamma^{(r-1)})],
\]
respectively. Observe that, modulo $ I $, we have the equality
\[
f[x,\gamma^{(1)},\dots,\gamma^{(r-2)}+\nu_{r-1}(\gamma^{(r-1)})]=\lambda^{(r-1)}f[x,\gamma^{(1)},\dots,\gamma^{(r-1)}].
\]
Since $\lambda^{(r-1)}$ is not a zero divisor in $\mathcal O_{U_\alpha}/ I $, this implies that the defining ideal sheaf of $K_r$ as a complex subspace of $p^{-1}(K_{r-1})$ is generated in $(U_\alpha,\varphi_\alpha)$ by the classes of the coordinate functions of $f[x,\gamma^{(1)},\dots,\gamma^{(r-1)}].$
\end{proof}
	\begin{ex}\label{exGenItDivDiffCor2} The double points $K_2$ of a map $f\colon \C^2\to \C^p$ are given by the vanishing of the divided differences
\begin{itemize}
\item[]$f_{1}[(x,y),\gamma^{(1)}]=\dfrac{f(x+\lambda,y+\lambda a)-f(x,y)}{\lambda},$ on the open subset $U_1$,
\item[]$f_{2}[(x,y),\gamma^{(1)}]=\dfrac{f(x+\lambda a,y+\lambda)-f(x,y)}{\lambda },$ on the open subset $U_2$.
\end{itemize}
If we wish to study triple points, then we will need to add the open subset $U_3$, where double points are given by the vanishing of
\begin{itemize}
\item[]$f_{3}[(x,y),\gamma^{(1)}]=\dfrac{f(x+\lambda ,y+\lambda (a-1))-f(x,y)}{\lambda }.$
\end{itemize}
The triple points of $f$ on $U_{11}$ are given by the vanishing of  $f_{1}[(x,y),\gamma^{(1)}]$ and the second divided difference 
\begin{align*}
&f_{11}[(x,y),\gamma^{(1)},\gamma^{(2)}]=\dfrac{f_1[(x,y),( \lambda + \mu, a + \mu b)]-f_1[(x,y),( \lambda , a )]}{ \mu }\\
&=\dfrac{\dfrac{f(x+ \lambda + \mu ,y+( \lambda +\mu)( a + \mu b))-f(x,y)}{ \lambda + \mu }-\dfrac{f(x+ \lambda ,y+ \lambda  a )-f(x,y)}{ \lambda }}{ \mu }.
\end{align*}

On $U_{12}$, triple points are given by the vanishing of $f_{1}[(x,y),\gamma^{(1)}]$ and 
\begin{align*}&f_{12}[(x,y),\gamma^{(1)},\gamma^{(2)}]=\dfrac{f_1[(x,y),(\lambda+ \mu  b , a+ \mu )]-f_1[(x,y),(\lambda, a)]}{ \mu }\\
&=\dfrac{\dfrac{f(x+ \lambda + \mu b,y+( \lambda + \mu b)( a+ \mu ))-f(x,y)}{ \lambda + \mu b}-\dfrac{f(x+ \lambda ,y+\lambda a)-f(x,y)}{ \lambda }}{ \mu },
\end{align*}
and similarly for the remaining open subsets $U_{21},U_{22},U_{31}$ and $U_{32}$ of the covering of $B_3$. This formulas are used explicitly in Example \ref{exDoblesTrifoldCone}.
	\end{ex}


	\begin{definition}
We say that $K_r$ is \emph{dimensionally correct} if $\dim K_r=nr-p(r-1)$.
	\end{definition}

	\begin{cor}\label{cor:LCI}
Let $X^n \to Y^p$ be a map between complex manifolds. The dimension of $K_r$ is at least $nr-p(r-1)$ at any point. If $K_r$ is dimensionally correct, then it is locally a complete intersection.
\begin{proof}This follows from the fact that $K_r$ is locally defined by $p(r-1)$ equations in $B_r$
\end{proof}
	\end{cor}

We finish this section by explaining how the computations can be simplified for maps of lower corank, and giving an example. Recall that two maps $f\colon X\to Y$ and $f'\colon X'\to Y'$ between manifolds are called \emph{$\cA$-equivalent} if there exist two biholomorphisms $\phi\colon X\to X'$ and $\psi\colon Y\to Y'$, such that $f'=\psi\circ f'\circ \phi^{-1}$. As a consequence of Theorem \ref{thmKrAsIntersection}, we obtain the following:

	\begin{prop}\label{propMrA-invariant}
If $f$ and $f'$ are $\cA$-equivalent maps, then $K_r(f)\cong K_r(f')$ via an isomorphism $B_r(X)\cong B_r(X')$ induced by $\phi$.
	\end{prop}

A map $F\colon T\times X\to T\times Y$  of the form $F(x,t)=(t,f_t(x))$ is said to be an \emph{unfolding} of the map $f_{t_0}\colon X\to Y$, for each $t_0\in T$. The manifold $T$ is called the parameter space. Recall that, by Proposition \ref{propKrUnfolding}, the multiple point space $K_r(F)$ is canonically embedded in $T\times B_r(\C^n)$, and $K_r(F)\cap \{t=t_0\}=K_r(f_{t_0})$. One checks easily that $K_r(F)$ is computed as follows:

	\begin{prop}\label{prop:unfoldings}
Let $F\colon T\times \C^n\to T\times \C^p$ be an unfolding of the form $F(t,x)=(t,f_t(x))$ and fix a covering collection of $(\P^{n-1})^{r-1}$. In each of the open subsets $T\times U_\alpha$, $\alpha\in \cS_r$, the multiple point space $K_r(F)$ is given by the vanishing of  
\[f_t[x,\gamma^{(1)}],\dots, f_t[x,\gamma^{(1)},\dots,\gamma^{(r-1)}].\]
	\end{prop}
	
	\begin{definition}
In the setting above, we call $f_t[x,\gamma^{(1)},\dots,\gamma^{(s)}]$ the \emph{relative divided differences} of $f_t(x)$.
	\end{definition}

\begin{rem}
	 If $f$ has corank $k$ at $x$ and $\dim X=n$, then locally $f$ is $\cA$-equivalent to an $(n-k)$-parameter unfolding. In particular, the relative divided differences of double and triple points of maps of corank two maps are as in Example \ref{exGenItDivDiffCor2}.
\end{rem}
 For a corank one map germ $f\colon (\C^n,0)\to (\C^p,0)$ we obtain the normal form
		\[(x_1,\dots, x_{n-1},y)\mapsto (x_1,\dots,x_{n-1},f_n(x,y),\dots,f_p(x,y)).\] 
The multiple points $K_r(f)$ of a map in such form can be embedded in $\C^{n-1}\times B_r(\C)=\C^{n-1}\times \C^r$. In this case, our atlas consists on a single chart, and the relative divided differences are the following expressions:
\begin{align*}
f[x,y,y^{(1)}]&=\dfrac{f(x,y^{(1)})-f(x,y)}{y^{(1)}-y},\\
f[x,y,y^{(1)},y^{(2)}]&=\dfrac{f[x,y,y^{(2)}]-f[x,y,y^{(1)}]}{y^{(2)}-y^{(1)}},\\
&\ \ \vdots\\
f[x,y,\dots, y^{(r-1)}]&=\dfrac{f[x,y,y^{(1)},\dots,y^{(r-3)},y^{(r-1)}]-f[x,y,y^{(1)},\dots,y^{(r-2)}]}{y^{(r-1)}-y^{(r-2)}}.
\end{align*}
These expressions were introduced by Marar and Mond's in \cite{MararMondCorank1} as equations for their multiple point space $D^r(f)$ of a corank one map. Consequently, we obtain the following result.

\begin{cor}\label{corItDivDiff}
 If $f$ is a corank one map, then $K_r(f)$ and $D^r(f)$ are equal.
\end{cor}

For arbitrary corank, a space $D^2(f)\subseteq X\times X$ was introduced by Mond in \cite{MondSomeRemarks}. A general construction of  multiple point spaces $D^r(f)\subseteq X^r$ was given by the authors in \cite{Nuno-Ballesteros2015On-multiple-poi}.

\begin{ex}\label{exDoblesTrifoldCone}
We are going to compute the spaces $K_2$ and $K_3$ of the map $f\colon\C^3\to\C^4$ given by
\[
(t,x,y)\mapsto(t,x^2+ty,y^2-tx,x^3+y^3+xy).
\]
The map $f$ is a one-parameter unfolding and, topologically, the real versions of the maps $f(0,x,y)$ and  $f(\epsilon,x,y),\epsilon\neq 0$ are as depicted in Figure \ref{figC2xC2Orbitmap}.
 \begin{figure}
\begin{center}
\includegraphics[scale=0.8]{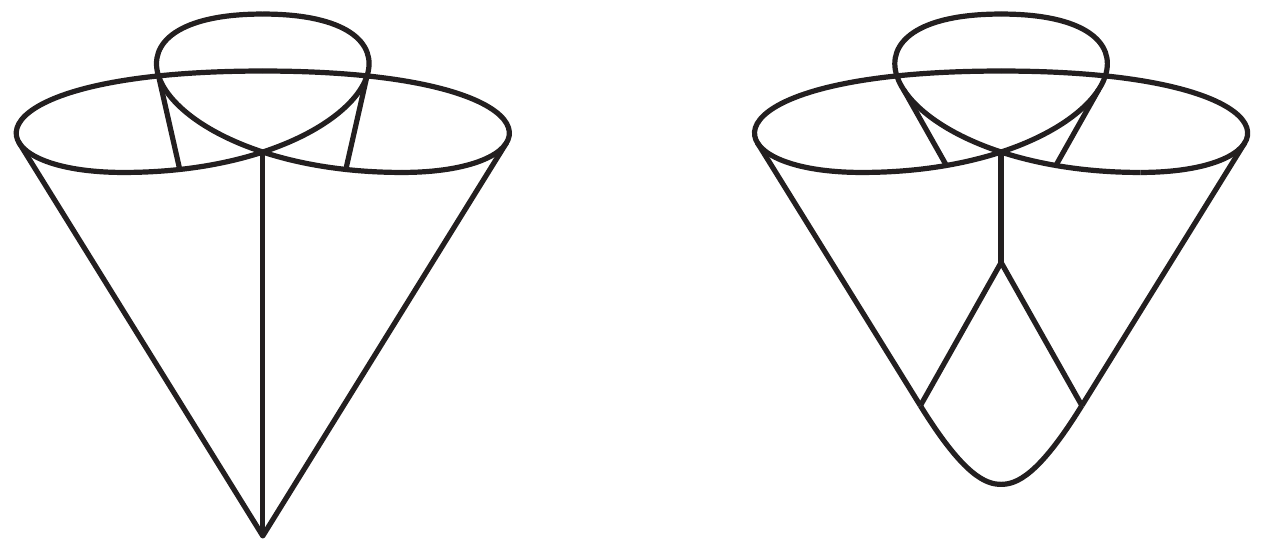}
\end{center}
\caption{Maps unfolded by $f$.}
\label{figC2xC2Orbitmap}
\end{figure}
One checks easily that $f$ has an isolated point of corank 2 at the origin.  Note the triple point in the image of the generic map $f(\epsilon,x,y)$, collapsing to the origin as $\epsilon$ tends to zero. 
 
   Since $f$ is an unfolding, Proposition \ref{prop:unfoldings} ensures that we may compute $K_r$ as a subspaces of $\C\times B_r(\C^2)$ by means of the relative divided differences. As already mentioned, the expressions from Example \ref{exGenItDivDiffCor2} compute double and triple points of maps of corank two, leaving $t$ as a parameter. In what follows, the notation for the atlas and divided differences is taken from there.

In the chart $\varphi_1\colon U_1\to\C^5$, with coordinates $(t,x,y,\lambda,a)$, the equations for $K_2\cap U_1$ are the vanishing of the divided differences of $f_2,f_3$ and $f_4$:
\begin{align*}\label{div-dif-trifold}
0&=a t+\lambda+2 x\\
0&=a^2 \lambda+2 a y-t\\
0&=a^3 \lambda^2+3 a^2 \lambda y+a \lambda+a x+3 a y^2+\lambda^2+3 \lambda x+3x^2+y.
\end{align*}
The computations can be performed with the library  {\tt IteratedMultPoint.lib} for  {\sc Singular}, by means of the sequence of commands
\begin{itemize}
\item[]{\tt LIB IteratedMultPoint.lib;}
\item[]{\tt ring r=0,(t,x,y),dp;}
\item[]{\tt list f=t,x2+ty,y2-tx,x3+y3+xy;}
\item[]{\tt ring S2=ItMP(f,2);}
\item[]{\tt L[1];}
\end{itemize}
The  instruction {\tt ItMP(f,2)} returns a ring with variables {\tt t,x,y,l(1),a(1)}, containing a list {\tt L}. The first entry {\tt L[1]} has the equations of $K_2$ on the open subsets $U_1$. The variables {\tt l(1),\tt a(1)} correspond to $\lambda$ and $a$. The space $K_2\cap U_1$ has dimension $2$ and thus $M_2$  is dimensionally correct and a complete intersection on on $U_1$. The projection $K_2\cap U_1\to \C\times(\C^2)^2,$ with coordinates $(t,(x,y),(x',y'))$, is described by 
\[x'=x+\lambda,\quad y'=y+\lambda a.\]

In order to cover $K_2$,  the divided differences on $U_2$ must be computed. We omit them, as they do not yield anything new. The equations for $K_2\cap U_2$ on {\sc Singular} are the content of the entry {\tt L[2]}.

Now we move to the computation of triple points $K_3$. As part of the process for triple points, we must compute double points in the extra open subset $U_3$. Again, this computation is uninteresting and omitted. 


We proceed now to the computation of the triple point space $K_3$ on the chart $\varphi_{11}\colon U_{11}\to\C^7$ of $\C\times B_3(\C^2)$, with coordinates $(t,x,y,\lambda,a,\mu,b)$. To the three equations  for $K_2$ on $U_1$ found above,  the vanishing of the following second divided differences must be added:
\begin{align*}
0&=b t+1,\\
0&=a^2+2 a b \lambda+2 a b \mu+b^2 \lambda \mu+b^2 \mu^2+2 b y,\\
0&=2 a^3 \lambda+a^3 \mu+3 a^2 b \lambda^2+6 a^2 b \lambda \mu+3 a^2 b \mu^2+3 a^2 y+3 a b^2 \lambda^2 \mu\\
&+6 a b^2 \lambda \mu^2+3 a b^2 \mu^3+6 a b \lambda y+6 a b \mu y +a+b^3 \lambda^2 \mu^2+2 b^3 \lambda \mu^3\\
   &+b^3 \mu^4+3 b^2 \lambda \mu y+3 b^2 \mu^2 y+b \lambda+b \mu+b x+3 b y^2+2 \lambda+\mu+3.
\end{align*}

These triple points are computed on {\sc Singular} by means of the sequence
\begin{itemize}
\item[]{\tt LIB IteratedMultPoint.lib;}
\item[]{\tt ring r=0,(t,x,y),dp;}
\item[]{\tt list f=t,x2+ty,y2-tx,x3+y3+xy;}
\item[]{\tt ring S3=ItMP(f,3);}
\item[]{\tt L[1][1];}
\end{itemize}

This time {\tt ItMP(f,3)} returns a ring with variables {\tt t},{\tt x},{\tt y},{\tt l(1)},{\tt l(2)},{\tt a(1)},{\tt a(2)}, together with a list whose entries {\tt L[}$i${\tt ]}{\tt[}$j${\tt ]} contain equations for the spaces $K_3\cap U_{ij}$, for $1\leq i\leq 3$ and $1\leq j \leq 2$.

\textsc{Singular} can also be used to check that $K_3\cap U_{11}$ has dimension 1, and hence it is dimensionally correct and a complete intersection.
Explicit equations for the projection of $K_3\cap U_{11}$ on $\C\times(\C^2)^3$ are obtained by putting
\[x'=x+\lambda,\quad y'=y+\lambda a,\quad x''=x+\lambda+\mu,\quad y''=y+(\lambda+\mu)(a+\mu b).\]

The situation changes when we compute $K_3$ on $U_{12}$. On this open subset, the first divided differences are the same, but the iterated ones are
\begin{align*}
0&=2 x+\lambda \mu+\mu^2 b+2 \lambda a +2 \mu a  b+a ^2 b\\
0&=-t+b\\
0&=3 x^2+3 x \lambda \mu+3 x \mu ^2 b+6 x \lambda a +6 x \mu  a  b+3 x a^2 b+3 y b
+y\\
&+\lambda^2 \mu ^2+2 \lambda \mu ^3 b+\mu ^4 b^2+3 \lambda^2 \lambda a +6 \lambda \mu ^2 a  b+3 \mu ^3 a b^2+3 \lambda^2 a ^2\\
&+6 \lambda \mu a ^2 b+3 \mu ^2 a ^2 b^2+2 \lambda a ^3 b+2 \lambda b+\lambda+\mu  
a ^3 b^2+\mu  b^2+\mu  b+a  b
\end{align*}
The equations of $K_3$ on  $U_{12}$ are contained in {\tt L[1][2]} and a computation with \textsc{Singular} shows that $K_3\cap U_{12}$ has dimension two. As a consequence, $K_3$ is not dimensionally correct. The image of the projection $K_3\cap U_{12}\to \C\times(\C^2)^3$ is obtained by putting
\[x'=x+\lambda,\quad y'=y+\lambda a,\quad x''=x+\lambda+\mu b,\quad y''=y+(\lambda+\mu b)(a+\mu).\]

Somehow surprisingly, the images of $K_3\cap U_{11}$ and $K_3\cap U_{12}$ on $\C\times(\C^2)^3$ are the same, despite coming from spaces of different dimensions (again, this can be checked with {\sc Singular}). A moment of thought will convince the reader of the fact that this implies that $K_3$ has an irreducible component contained in the exceptional divisor of $B_3\to B_2\times_{\C^3}B_2$. This and related pathologies are explained in the next section.
\end{ex}

\begin{rem}
When using the library {\tt IteratedMultPoint.lib} on an $s$-parameter unfolding $f$, it is convenient to introduce the $s$ parameters in the front of the list of polynomials defining $f$. This way, the procedure {\tt ItMP(f,r);} makes computations in $\C^s\times B_r(\C^{n-s})$, as indicated in Proposition \ref{prop:unfoldings}. If, for example, we were to reorder the coordinate functions of the previous example as {\tt list f=x2+ty,y2-tx,x3+y3+xy,t}, the equations will be given in $B_r(\C^3)$. The equations and the coverings would still be correct, but more complicated.
\end{rem}

\begin{rem}The computation of {\tt ItMP(f,r)} involves choosing a covering collection for $(\P^{n-1})^{r-1}$, which the procedure does internally. A different collection will give the same space $K_r$, but may result in very different covering and equations.
\end{rem}

									\section{Pathologies}\label{secPathologies}
									
As Kleiman observes in \cite{KleimanMultiplePointFormulasI}, the idea that $K_r$ is the double point space of $K_{r-1}\to K_{r-2}$ is just a definition, with a clear interpretation only for strict multiple points. As it tuns out, with the presence of points of corank $\geq 2$ the iteration principle may yield too many points, and may do so in a non symmetrical way. These pathologies come as no surprise; the excess of dimension and the fact that $K_r$ and the target multiple points disagree are somehow easy set-theoretical considerations, while the lack of symmetry was already pointed out  by Ran in \cite[Section 1]{Ran1985}. Our explicit description of $K_r$ just allows us to be more precise about them, which will be crucial for results in sequel of this work \cite{SecondIterated}.

\subsection*{Excess of dimension}
Assume that a map $f\colon X\to Y$ between manifolds has corank $\geq 2$ at $x\in X$. Following the description of double points in Proposition \ref{propK2ForSmooth}, we may take any two different points $u^{(1)},u^{(2)}\in \P (\ker df_x)$ to produce two double points $(x,x,u^{(1)})$ and $(x,x,u^{(2)})$. Since the map $B_2\to X$ drops the $u^{(i)}$, these two points form a point in $K_2\times_{K_1}K_2$ away from $\Delta B_2$. This point is the image of a point in $B_3\to B_2\times_{B_1}B_2$, which locally is an isomorphism. Since this preimage is not contained in the exceptional divisor, it is contained in $K_3=b^{-1}(K_2\times_{K_1} K_2):b^{-1}(\Delta B_2)$. 

 Summarising, any pair of different points $u,u'\in \P(\ker df_x)$ produces a triple point, with no further conditions on $u, u'$. The argument carries on to higher multiple spaces; the following result counts exactly how many points are obtained that way. 

	\begin{prop}\label{propDimDiagFiberCorank}
Let $f$ be a map between manifolds, let $x$ be a point where the corank of $f$ is $k\geq 2$ and let $r\geq 2$. The preimage of $(x,\dots,x)\in X^r$ by the map $K_r\to X^r$ has dimension $(r-1)(k-1)$.
\begin{proof}First we show that the dimension is at least $(r-1)(k-1)$, by showing that any tuple $(u^{(1)},\dots,u^{(r-1)})$ of different points $u^{(i)}\in \P (\ker df_x)$ can be identified with a point in $K_r$ mapping to $(x,\dots,x)$.  The case of $r\geq 2$ was shown above. The case of $r\geq 3$ is analogous since, as one can check easily, a tuple $(u^{(1)},\dots,u^{(r-1)})$ with all $u^{(i)}$ different determines a unique point in $B_r$, and the map $B_r\to B_{r-1}$ just drops the last component $u^{(r-1)}$.

Now we show that the dimension is at most $(r-1)(k-1)$. As a consequence of Proposition \ref{propMrAndPartitions}, it suffices to show the claim for a map germ $f\colon (\C^n,0)\to (\C^p,0)$ of corank $k$, which by Proposition \ref{propFunctorialityK_k} may be taken of the form 
\[f(x)=(x_1,\dots,x_{n-k},h(x)).\]
Following Proposition \ref{propKrUnfolding},the space $K_r$ can be embedded in $\C^{n-k}\times B_r(\C^k)$. Under this embedding, the points mapping to $(x,\dots,x)$ we are looking for now correspond to points mapping to 
\[(x_1,\dots,x_{n-k},(y,\dots,y)),\] where $y=x_{n-k+1},\dots, x_{n}$. Therefore, the dimension we want to compute is at most the dimension of the preimage of $(y,\dots,y)$ by the map $B_r(\C^k)\to (\C^k)^r$. In any of the charts from Section \ref{secCoordinatesForBr}, this set is given by $\lambda^{(i)}=0$. The free $a^{(i)}$, with $i=1,\dots,r-1$, give a fiber of dimension $(r-1)(k-1)$.
\end{proof}
	\end{prop}
\begin{cor}
Let $f\colon X\to Y$ between complex manifolds, with points of corank $k\geq 2$, and assume that $\dim X-\dim Y\leq k$. Then there exists $r_0$ such that $K_r$ is not dimensionally correct, for all $r\geq r_0$.
\end{cor}

\subsection*{$\boldsymbol{K_r}$ and target multiple points}Kleiman, Lipman and Ulrich \cite{Kleiman1996The-multiple-po} studied relations between the multiple point spaces given by iteration and by the Fitting ideals.  For any finite map $f\colon X\to Y$, the subspace $N_r(f)\subseteq Y$ is given by the vanishing of the $r-1$ Fitting ideal of $f_*\cO_X$. The image of $f$ is $N_1(f)$, the double points of $f$ in $Y$ are $N_2(f)$, and so on. Pulling back this points, we obtain the multiple point spaces $M_r(f)=f^{-1}(N_r(f))\subseteq X$. Here, $X$ and $Y$ are not assumed to be smooth, and a suitable extended definition of corank is used. Write 
\[K_{r+1}\stackrel{f_r}{\longrightarrow} K_{r}\] for the usual maps, and  
\[K_{r}\stackrel{\rho_{r}}{\longrightarrow} X\] the maps obtained by composition. 

For finite maps $f\colon X\to Y$ of corank one, with $\dim Y=\dim X+1$, they show that
\[N_{r-1}(f_1)=M_r(f)\]
and that all $f_r$ are finite maps of corank one. In this case, the maps $\rho_{r}$ are also finite and have corank one, and one obtains the set-theoretical equality 
\[N_1(\rho_r)=M_r(f).\] This means that the projection $\rho_r(K_r)$ of the iterated multiple points $K_r$ is the same as the inverse image $f^{-1}(N_r(f))$ of the target multiple points. 

The first problem that one encounters in the case of corank $\geq 2$ is that the maps $f_r$, and hence the $\epsilon_r$ are not finite anymore. Therefore their pushforward modules are not finitely presented modules, their Fitting ideals are not defined and it is not clear how one should define the algebraic structure of the projections $\rho_r(K_r)$. But the problem is worse, as  $\rho_r(K_r)$ and $f^{-1}(N_r(f))$ do not agree even at the set-theoretical level. To see this, just observe that $N_r(f)$ is empty, for every $r$ bigger than the multiplicity of $f$, while $K_r$ is never empty in the presence of points of corank $\geq 2$, as a consequence of  Proposition \ref{propDimDiagFiberCorank}.

\subsection*{Lack of symmetry}
It is well known that the multiple point spaces $D^r\subseteq X^r$ are invariant by the action of the symmetric group $S_r$ by permutation of the coordinates in $X^r$. It would be reasonable to expect the spaces $B_r$ to have natural actions, lifting those on $X^r$, and that these actions restrict to actions on $K_r$ compatible with those on $D^r$. We show that, unfortunately, this is not the case for $r\geq 3$. Before this, we show that, for a manifold $X$, the action of $S_2$ on $X^2$ can be lifted to $B^2$, and indeed there is a unique way of doing it.

Let $X$ be a complex manifold and $\sigma\colon X^2\to X^2$ be the map corresponding to the transposition $(1\;2)$. To lift the action of $S_2$ on $B_2$, we must find an involution $\tilde\sigma\colon B_2\to B_2$, making the diagram 
\[
\begin{tikzcd}[column sep=small]
B_2 \arrow[r,"\tilde\sigma"] \arrow{d}& B_2\arrow{d}\\
X^2\arrow[r,"\sigma"]						& X^2
\end{tikzcd}
\]
commutative.  For this we use the fact that  $B_2$ is the blowup of $X^2$ along $\Delta X$. Let $h$ be the composite map $B_2\to X^2\stackrel{\sigma}{\longrightarrow} X^2$. Since the preimage of $\Delta X$ by $\sigma$ is again $\Delta X$, the preimage $h^{-1}(\Delta X)$ is the exceptional divisor $E$ of $B_2$. The universal property of the blowup $B_2\to X^2$ applies to $h$ and gives a unique map $\tilde\sigma$ satisfying the above commutativity. The fact that $\tilde \sigma$ is an involution follows from the fact that $\tilde \sigma$ can be identified with $\sigma$ on $B_2\setminus E$, a dense subset of $B_2$.

Now we show that, for a complex manifold $X$ of dimension at least $2$, there is no action of $S_3$ on $B^3$ lifting the action on $X^3$. We assume for simplicity that $X=\C^n$. Since the transposition $(1\;2)$ takes the set $\Delta_{13}=\{(x,x',x)\in X^3\}$ to $\Delta_{23}=\{(x',x,x)\in X^3\}$, it suffices to show that the fibers in $B_3$ of $\Delta_{13}$ and $\Delta_{23}$ are not isomorphic. We start by looking at their fibers the space $B_2\times_X B_2$, which consists of points 
\[\left.\begin{array}{llllll} 
{u'}     &u  &  \\
 x''   &x'  &x
\end{array}\right.
\]
with $(x''-x)\wedge u'=0=(x'-x)\wedge u.$
At an open subset given by, say, $u'_i\neq 0\neq u''_i$, the fibre of $\Delta_{12}$ is given by the equation $x'_i-x_i=0$ and hence is a Cartier divisor. However, the fibre of $\Delta_{23}$ equals $Z\cup \Delta B_2$, where the diagonal $B_2$ is given by $x''_i=x'_i$ and $u'=u$, and $Z$ is a codimension 2 component given by $x'_i=x''_i=0$, which is not contained in $\Delta B_2$ because $n\geq 2$.

On one hand, it is clear that the fibre of $\Delta_{13}$ is a divisor in $B_3$, because it is the preimage of the corresponding fibre in $B_2\times_X B_2$, which is already a divisor. On the other hand, the space $B_3$ is obtained by blowing up $\Delta B_2$, and therefore the preimage of $\Delta B_2$ is also a divisor in $B_3$. However, since the structure map of the blowup is an isomorphism away from the exceptional divisor, and $Z$ is not contained in $\Delta B_2$, we conclude that the fibre of $D_{13}$ has a component of codimension 2 arising from $Z$. This shows that the fibres of $\Delta_{13}$ and $\Delta_{23}$ are not isomorphic.

									\appendix

								\section{Fibered products and intersections}\label{appendixIntersections}
We overview some basics about preimages, intersections and diagonals in a category $\cC$  with fibered products. Since the proofs are simple and involve very few ingredients, most of them are left to the reader.  In what follows, an arrow $X\hookrightarrow Y$ stands for a monomorphisms in $\cC$. The usual commutative diagram involving a fibered product is called a cartesian square.
 
 \begin{prop}\label{propMonomorphismsInFibProdDiag}
Consider a cartesian square
 \[
\begin{tikzcd}
X_1\times_YX_2 \arrow{r} \arrow{d}& X_1\arrow{d}\\
X_2\arrow{r}						& Y 
\end{tikzcd}
\]
If $X_1\to Y$ is a monomorphism (resp. isomorphism), then $X_1\times_YX_2\to X_2$ is a monomorphism (resp. isomorphism).
\end{prop}

\begin{definition}\label{defDiagonalObject}\label{defPreimageObject}\label{defIntersectionObject}The \emph{diagonal} of an object $X$ is $\Delta X=X\times_X X.$
Given a morphism $f\colon X\to Y$ and a monomorphism $S\hookrightarrow Y$, the \emph{preimage of $S$ by $f$} is
$f^{-1}(S)=X\times_YS.$
Given two monomorphisms $X_i\hookrightarrow \cX$, the \emph{intersection} of $X_1$ and  $X_2$ is $X_1\cap X_2=X_1\times_\cX X_2.$
\end{definition}

 By Proposition \ref{propMonomorphismsInFibProdDiag}, the preimage comes equipped with a monomorphism $f^{-1}(S)\hookrightarrow X$ and a morphism $f^{-1}(S)\to S$, called the \emph{restriction of $f$}. Similarly, the morphisms $X_1\cap X_2\to X_i$ in the cartesian square of an intersection are monomorphisms.

\begin{prop}
 Given two commutative diagrams 
\[
\begin{tikzcd}
X_i \arrow{r} \arrow[hookrightarrow]{d}& Y\arrow{d}\\
\cX_i\arrow{r}						& Y' 
\end{tikzcd}
\]
 there is a unique  monomorphism $X_1\times_YX_2\hookrightarrow \cX_1\times_{Y'}\cX_2,$ compatible with $X_i\hookrightarrow \cX_i$ in the obvious way. Moreover,
if the morphisms $X_i\hookrightarrow \cX_i$ are isomorphisms and $Y\to Y'$ is a monomorphism, then the above morphism is an isomorphism between $X_1\times_YX_2$ and $\cX_1\times_{Y'}\cX_2$.
\end{prop}

	\begin{prop}
The diagonal satisfies the following properties:
\begin{enumerate}
\item For any morphism $X\to Y$, there is a unique monomorphism $\Delta X\hookrightarrow X\times_Y X,$ compatible with the projections from fibered products. 
\item A morphism  $X\to Y$ is a monomorphism if and only if $\Delta X\hookrightarrow X\times_Y X$ is an isomorphism.
\item For any commutative diagram
\[
\begin{tikzcd}
X \arrow{r} \arrow[hookrightarrow]{d}& Y\arrow{d}\\
\cX\arrow{r}						& Y' 
\end{tikzcd}
\]
Inside $\cX\times_{Y'}\cX$, we have that $\Delta X=\Delta\cX\cap(X\times_Y X)$. Equivalently $\Delta X=e^{-1}(\Delta\cX)$, for the induced monomorphism $e\colon X\times X\hookrightarrow \cX\times \cX$.
\end{enumerate}
	\end{prop}

In the case of intersections, the universal property of the fibered product takes an easier form:

\begin{prop}[Universal property of the intersection]\label{propUnivPropIntersectionObject}  Let $P$ be an object sitting in a commutative diagram 
\[
\begin{tikzcd}
P \arrow{r} \arrow{d}& X_1\arrow[hookrightarrow]{d}\\
X_2\arrow[hookrightarrow]{r}						& \cX
\end{tikzcd}
\]
 Then $P$ is the intersection of $X_1\cap X_2$ if and only if the following two hold:
\begin{enumerate}
\item One of the $P\to  X_i$ is a monomorphism.
\item For any object $P'$, equipped with morphisms $P'\to X_i$ commuting with $X_i\hookrightarrow \cX$,  there is a morphism $P'\to P$ making commutative one of the diagrams 
\[
\begin{tikzcd}[row sep=1.3em,column sep=0.3em]
P' \arrow[rr,hookrightarrow]	\arrow[dr]&	& X_i \\
	& P\arrow[ur,hookrightarrow]&
\end{tikzcd}
\]
\end{enumerate}
\begin{proof}We see that $P$ satisfies the universal property of $X_1\times_\cX X_2$ in three easy steps left to the reader: First, observe that if one of the  $P\to  X_i$ is a monomorphism, then the other is a monomorphism as well. Second, use the first observation to check that if one of the diagrams in the second item commutes, then so does the second. Finally, the unicity of the morphism $P'\to P$ follows from any of the commutativities in the second item.
\end{proof}
\end{prop}

									\section{Some proofs}\label{appendixProofs}
	Three basic lemmas will be used. The first two of them follow immediately from the definition of the symmetric algebra.
\begin{lem}\label{lemMorphExtensionToSym}
Any morphism $M\to N$ of $R$-modules extends uniquely to a morphism of $R$-algebras $S(M)\to S(N).$
Moreover, If $M\to N$ is an epimorphism (resp. isomorphism), then $S(M)\to S(N)$ is an epimorphism (resp. isomorphism).
\end{lem}

\begin{lem}\label{lemMorphSymOverDiffrentBasis}
For any ring morphism $R'\to R$ and any $R$-module $M$, there is unique graded morphism of algebras $S_{R'}(M)\to S_R(M),$
which is $R'\to R$ in degree zero and $\id_M$ in degree one. For all $d\geq 1$, the degree $d$ part $(S_{R'}(M))_d\to (S_R(M))_d$ is an epimorphism, and it is an isomorphism  
if $R'\to R$ is an epimorphism.
\end{lem}

\begin{lem}\label{lemMorphismToRelativeProj}
Let $\sM$ be a quasi-coherent $\cO_T$-module and $Y=\Proj_TS(\sM)$. Let $X\stackrel{g}\to T$ be another scheme over $T$. To give a morphism
\[
\begin{tikzcd}[row sep=1.3em,column sep=0.3em]
X \arrow[rr]	\arrow[dr]&	& Y\arrow[dl] \\
	& T&
\end{tikzcd}
\]
is equivalent to give
\begin{enumerate}
\item An invertible sheaf $\sL$ on $X$,
\item An epimorphism $\psi\colon  g^*\sM\twoheadrightarrow \sL$ of $\cO_T$-modules.
\end{enumerate}
Two pairs $(\sL,\psi)$ and $(\sL',\psi')$ determine the same morphism if and only if there exists an isomorphism of $\cO_T$-modules $\sL\stackrel{\sim}{\to}\sL'$, such that the following diagram commutes:
\[
\begin{tikzcd}[row sep=1.3em,column sep=0.3em]
&\sM	\arrow[dl]\arrow[dr]& \\
\sL	\arrow[rr,"\sim"]& &\sL'
\end{tikzcd}
\]
\begin{proof} This is a particular case of
\cite[Tag 01O4]{stacks-project}, after observing that $\psi$ extends uniquely to an epimorphism $g^*S(\sM)\twoheadrightarrow S(\sL)$, by Lemma \ref{lemMorphExtensionToSym},  and that $\bigoplus_n \sL^{\otimes n}=S(\sL)$, because $\sL$ is invertible. \end{proof}
\end{lem}

\begin{deferredproof}[Proposition \ref{propResidualOfSubspaceAlongIntersection}]\label{proofResidualOfSubspaceAlongIntersection} Given $e\colon X\hookrightarrow \cX$, let $W=\cW\cap X=e^{-1}(\cW)$. We have an epimorphism \[\cO_\cX\stackrel{e^*}{\twoheadrightarrow}\cO_X,\] taking $I_\cW$ onto $I_W$. The fact that $W=e^{-1}(\cW)$ means that $I_W$, the ideal of $W$, is generated over $\cO_X$ by $e^*(I_\cW)$, where $I_\cW$ is the ideal of $\cW$. Since $e^*$ is onto, this in turn means that $e^*$ restricts to an epimorphism $I_\cW\twoheadrightarrow I_W$ of $\cO_\cX$-modules. By Lemma \ref{lemMorphExtensionToSym}, this morphisms extends to an epimorphism
\[S_{\cO_\cX} (I_\cW)\twoheadrightarrow S_{\cO_\cX} (I_W).\]
By Lemma \ref{lemMorphSymOverDiffrentBasis}, there is also an epimorphism $S_{\cO_\cX} (I_W)\twoheadrightarrow S_{\cO_X}I_W$. Composing these two epimorphisms we obtain $S_{\cO_\cX} (I_\cW)\twoheadrightarrow S_{\cO_X} (I_W)$, which  induces the monomorphism $\Res_WX\hookrightarrow \Res_\cW\cX$. Functoriality follows from functoriality of the elements involved.
\end{deferredproof}

%
%

\begin{deferredproof}[Proposition \ref{propResAndIntersection}]\label{proofResAndIntersection}
Write
\[X=X_1\cap X_2,\quad W=W_1\cap W_2,\]
\[R=\Res_WX,\quad R_i=\Res_{W_i}X_i\quad\text{and}\quad \cR=\Res_{\cW}\cX.\] 
Writing $e_i\colon X\hookrightarrow X_i$, and $I_{W_i}, I_W$ for the ideals defining $W_i, W$ in $X_i,X$, respectively, one has the equalities 
\[I_W=e_i{}^*I_{W_i}.\]

We  show that $R$ satisfies the universal property of the intersection $R_1\cap R_2$ (see Proposition \ref{propUnivPropIntersectionObject}). First of all, by Proposition \ref{propResidualOfSubspaceAlongIntersection}, there are monomorphisms $R\hookrightarrow R_i$, commuting with $R_i\hookrightarrow \cR$ by functoriality. 
Now let $P$ be a complex space, equipped with two morphisms $a_i\colon P\to R_i$ commuting with $R_i\hookrightarrow \cR$. We must show that there exists a morphism $P\to R$, such that the following diagram is commutative:
\[
\begin{tikzcd}[row sep=1.3em,column sep=0.3em]
P \arrow[rr,"a_1"]	\arrow[dr]&	& R_1 \\
	& R\arrow[ur,hookrightarrow]&
\end{tikzcd}
\]
Let $g_i\colon P\to X_i$ be the composition $P\hookrightarrow R_i\to X_i$. The morphisms considered so far fit into the commutative diagram
\[\begin{tikzcd}[sep=0.8em]
&P\arrow[dddrr,]\arrow[ddddl]\arrow[dddddrr, bend left=85]\arrow[ddddddl,bend right=85]\\
\\
\\
	&R \arrow[rr,hookrightarrow] \arrow[dl,hookrightarrow] \arrow[dd] 	&	&  R_1 \arrow[dd] \arrow[dl,hookrightarrow] \\
 R_2 \arrow[rr,crossing over,hookrightarrow] \arrow[dd]&&  \cR& \\
	&X\arrow[rr,swap,hookrightarrow] \arrow[dl,hookrightarrow] && X_1 \arrow[dl,hookrightarrow] \\
 X_2 \arrow[rr,hookrightarrow] && \cX\arrow[from=uu,crossing over]
\end{tikzcd}
\]
The commutativity of $g_i$ with $X_i\hookrightarrow \cX$ implies the existence of a unique morphism 
$g\colon P\to X$, satisfying the equalities
\[g_i=e_i\circ g.\]

Observe that $\Res_{W_i}X_i$ is the relative homogeneous spectrum of a symmetric algebra over $\cO_{X_i}$, and that we have defined $g_i$ so that $a_i$ is compatible with $P\to X_i$ and $R_i\to X_i$. Therefore, by Lemma \ref{lemMorphismToRelativeProj}, the existence of $a_i$ implies the existence of an invertible sheaf $\sL_i$ on $P$ (up to isomorphism) and an epimorphism of $\cO_{X_i}$-modules
\[g_i^*I_{W_i}\twoheadrightarrow \sL_i.\]
In particular, since $e_1^*$ is an epimorphism $\cO_{X_1}\twoheadrightarrow\cO_X$ and we have the equalities $g_1^*=g^*\circ e_1^*$ and $e_1^*I_{W_1}=I_W$, this morphism determines an epimorphism of $\cO_X$-modules 
\[g^*I_W\twoheadrightarrow \sL_1.\]
This epimorphism in turn determines a morphism $P\to R$, satisfying the desired commutativity by construction.
\end{deferredproof}

\begin{deferredproof}[Lemma \ref{LemaKleiman}]\label{proofLemaKleiman}Applying Remark \ref{remResSubspaceRegEmb} to the inclusions $Y\cap Z\subset Z\subset X$, we obtain
 \[\Res_{Y\cap Z} Z=r^{-1}(Z):r^{-1}(Y\cap Z).\] We show that this space is isomorphic to $b^{-1}(Z):b^{-1}(Y)$.

Let $Y=V( I )$ and $Z=V(J)$, for some coherent ideal sheaves in $X$. Assume that $ I $ is generated in an affine open subset $U\subseteq X$ by a regular sequences $g_1,\dots,g_k\in\mathcal O_X(U)$, which may be completed to a regular sequence $g_1,\dots,g_k,h_1,\dots,h_r$ generating $ I +J$. We may also complete $h_1,\dots,h_r$ with some $p_1,\dots, p_l$ to get a set of generators of $J$.

The blowup $\Bl_YX$ in $U':=b^{-1}(U)$ is isomorphic to the complex subspace of $U\times \mathbb P^{k-1}$, defined by the vanishing of the $2\times 2$-minors of the matrix
$$
\left(
\begin{array}{ccc}
g_1  & \cdots  & g_k \\
u_1  & \cdots  & u_k
\end{array}
\right)
$$
and $b^{-1}(Z)$ is obtained in $U'$ just by adding the equations $h_1=\dots=h_k=0$ and $p_1=\dots=p_l=0$. For each $i=1,\dots,k$, let $U'_i$ be the affine open subset of $U\times \mathbb P^{k-1}$ given by $u_i=1$. The space  
$b^{-1}(Z)$ is defined in $U'_i$ by the ideal $A'=A'_1+A'_2$, where:
\begin{enumerate}
\item $A'_1$ is generated by $g_s-g_iu_s$, with $s=1,\dots,k$, $s\ne i$;
\item $A'_2$ is generated by $h_1,\dots,h_r$ and $p_1,\dots,p_l$.
\end{enumerate}
The preimage $b^{-1}(Y)$ is defined in $U_i'$ by the ideal $B'=A'_1+(g_i)$. Therefore
\begin{align*}
A':B'&=(A'_1+A'_2):(A_1'+(g_i))\\
&=\left((A'_1+A'_2):A_1'\right)\cap \left((A'_1+A'_2):g_i\right)\\
&=(A'_1+A'_2):g_i.
\end{align*}

Analogously, $\Bl_{Y\cap Z} X$  is defined in the open set $U'':=r^{-1}(U)$ as the closed complex subspace of $U\times \mathbb P^{k+r-1}$ given by the $2\times 2$-minors of the matrix
\[
\left(
\begin{array}{cccccc}
g_1  & \cdots  & g_k & h_1 & \cdots & h_r  \\
u_1  & \cdots  & u_k & v_1 & \cdots & v_r
\end{array}
\right)
\]
Here we cover $U\times \mathbb P^{k+r-1}$ by affine open subsets $U''_i=\{u_i=1\}$, with $i=1,\dots,k$, and $V''_j\{v_j=1\}$, with $j=1,\dots,r$. On one hand, the defining ideal of $r^{-1}(Z)$ in $U_i''$ is $A''=A''_1+A''_2+g_iV$, where:
\begin{enumerate}
\item $A''_1$ is generated by $g_s-g_iu_s$, with $s=1,\dots,k$, $s\ne i$;
\item $A''_2$ is generated by $h_1,\dots,h_r$ and $p_1,\dots,p_l$;
\item $V$ is generated by $v_1,\dots,v_r$.
\end{enumerate}
The defining ideal of $r^{-1}(Y\cap Z)$ in $U_i''$ is $B''=A''_1+A''_2+(g_i)$, hence:
\begin{align*}
A'':B''&=(A''_1+A''_2+g_iV):(A''_1+A''_2+(g_i))\\
&=(A''_1+A''_2+g_iV):g_i\\
&=((A''_1+A''_2):g_i)+V.
\end{align*}
We have an isomorphism
\[
\frac{\mathcal O_{U_i''}}{A'':B''}=\frac{\mathcal O_{U_i''}}{(A''_1+A''_2):g_i+V}\cong 
\frac{\mathcal O_{U_i'}}{(A'_1+A'_2):g_i}=\frac{\mathcal O_{U_i'}}{A':B'},
\]
 induced by the map 
\[\varphi:(b^{-1}(Z):b^{-1}(Y))\cap U''_i\to (r^{-1}(Z):r^{-1}(Y\cap Z))\cap U'_i,
\]
given by $(x,[u])\mapsto (x,[u,0])$ for $x\in U$ and $u\in \mathbb P^{k-1}$. It follows that this map is an isomorphism of complex spaces.

On the other hand, the defining ideal of $r^{-1}(Z)$ in $V_j''$ is generated by $g_1,\dots,g_k$ and $h_1,\dots,h_r$, which coincides with the defining ideal of $r^{-1}(Y\cap Z)$. It follows that
\[
\left(r^{-1}(Z):r^{-1}(Y\cap Z)\right)\cap V_j''=\emptyset.
\]
Therefore, the above isomorphisms from $(b^{-1}(Z):b^{-1}(Y))\cap U''_i$ to $(r^{-1}(Z):r^{-1}(Y\cap Z))\cap U'_i$ glue together to give an isomorphism of complex spaces
\[
\varphi\colon  (b^{-1}(Z):b^{-1}(Y)) \to (r^{-1}(Z):r^{-1}(Y\cap Z)).
\]
\end{deferredproof}

\begin{deferredproof}[Proposition \ref{propMrAndPartitions}]\label{proofMrAndPartitions} To use induction on $r$ we  prove a slightly more detailed result. We write $K_{r_i}$ for $K_{r_i}(f^{(i)})$, as the $f^{(i)}$ are clear from the context. Let $z\in K_r$ project to a point $w\in K_{r-1}$, via the map $K_r\to K_{r-1}$.  We claim that $K_{r-1}(f)$, around $w$, is locally isomorphic to 
\[K_{r_1}\times_Y\dots\times_Y K_{r_s},\]
for a partition $r_1,\dots,r_s=r-1$. But we also claim that one of the following statements holds:
\begin{enumerate}
\item $K_r$ is locally isomorphic to 
\[K_{r_1}\times_Y\dots\times_Y K_{r_s}\times_YK_1,\] and the map $K_r\to K_{r-1}$ drops the last component.
\item $K_r$ is locally isomorphic to 
\[K_{r_1}\times_Y\dots\times_Y K_{r_i+1}\times_Y\dots\times_YK_{r_s},\] with the map $K_{r}\to K_{r-1}$ being the restriction of \[\id_{K_{r_1}}\times\dots\times(K_{r_i+1}\to K_{r_i})\times\dots\times\id_{K_{r_s}}.\]
\end{enumerate}
 
 There is nothing to prove for $K_1$ and for $K_2$ there are two cases: if $z\in K_2$ is not contained in the exceptional divisor of $B_2$, then locally $K_2$ is isomorphic to $X\times_Y X$, that is, $K_2\cong K_1\times_Y K_1$, which is case (1). On the exceptional divisor we are looking at $K_2\to K_1$, which is case (2).
 
 Now assume that the statement holds up to $r\geq 2$ and compute $K_r\times_{K_{r-1}}K_r$, around a point whose first projection is $z\in K_r$ and its second projection is $z' \in K_r$. Consider the following cases:

If $K_r$ is of the form (1) around both $z$ and $z'$, then $K_r\times_{K_{r-1}}K_r$ is isomorphic to 
\[K_{r_1}\times_Y\dots\times_Y K_{r_s}\times_YK_1\times_YK_1.\]
We need to subdivide this case further: If $z=z'$, then the previous isomorphism  takes $\Delta K_r$ to $K_{r_1}\times_Y\dots\times_Y K_{r_s}\times_Y\Delta K_1.$ In this case, the space $K_{r+1}$ is locally isomorphic to 
\[K_{r_1}\times_Y\dots\times_Y K_{r_s}\times_YK_2.\] The map $K_{r+1}\to K_r$ is given by $K_2\to K_1$ on the last component and leaves the other components untouched. This is an instance of case (2). If $z\neq z'$, then the diagonal $\Delta K_r$ does not intersect the open subset of $K_r\times_{K_{r-1}}K_r$ we are looking at, as long as the open neighborhoods $U_i$ are small enough. In this case $K_{r+1}$ is locally isomorphic to $K_r\times_{K_{r-1}}K_r$ and the map $K_{r+1}\to K_r$ drops the last component. This falls into case (1).

If $K_r$ is of the form (1) around $z$ and of the form (2) around $z'$, then $K_r\times_{K_{r-1}}K_r$ is isomorphic to 
\[K_{r_1}\times_Y\dots\times_Y K_{r_i+1}\times_Y\dots\times_YK_{r_s}\times_YK_1,\] and the first projection to $K_r$ is given by $K_{r_i+1}\to K_{r_i}$ on the $i$th component. By hypothesis, these case occurs only for $z\neq z'$, hence $K_{r+1}$ is locally isomorphic to $K_r\times_{K_{r-1}}K_r$ and we are in the case (2).

We leave to the reader to check the following cases: If $K_r$ is of the form (2) around $z$ and of the form (1) around $z'$, then 
we are in case (1). When $K_r$ is of the form (2) around $z$ and $z'$,
%
 we are in case (2).
\end{deferredproof}

	\begin{deferredproof}[Lemma \ref{lemTriangleK}]\label{proofTriangleK}
We proceed by induction on $s$, but skip the case of $s=2$, as is analogous to the general case. Assume that the statement is true for $s-1$. We may use the induction hypothesis and usual commutativities for fibered products to obtain the equality of the following compositions:
\begin{align*}
\Big((K_{r+1}/K_{r})^s\to (K_{r+1}/K_{r})^{s-1}\to (K_{r}/K_{r-1})^{s}\to K_{r-1}\Big)&=\\
\Big((K_{r+1}/K_{r})^s\to (K_{r+1}/K_{r})^{s-1}\to (K_{r}/K_{r-1})^{s}\stackrel{\beta}{\to} K_r\to K_{r-1}\Big)&=\\
\Big((K_{r+1}/K_{r})^s\to (K_{r+1}/K_{r})^{s-1}\stackrel{\beta}{\to} K_{r+1}\to K_r\to K_{r-1}\Big)&=\\
\Big((K_{r+1}/K_{r})^s\stackrel{\beta}{\to} K_{r+1}\to K_{r}\to K_{r-1}\Big)&=\\
\Big((K_{r+1}/K_{r})^s\stackrel{\beta}{\to} K_{r+1}\stackrel{\tau}{\to} K_{r}\to K_{r-1}\Big)&.
\end{align*}

This means that the two external paths in the hexagon below are equal. Therefore, by the universal property of $(K_{r}/K_{r-1})^{s+1}$, there exists a unique map
$(K_{r+1}/K_{r})^s\to (K_{r}/K_{r-1})^{s+1}$, satisfying the desired commutativites.
\[
\begin{tikzcd}[row sep= 10,column sep=-2.85em]
(K_{r+1}/K_{r})^s \arrow[rr] \arrow[dr,dashrightarrow] \arrow[dd,"\beta"'] \arrow[dd] 	&	&  (K_{r+1}/K_{r})^{s-1}  \arrow[dr] \\
	& (K_{r}/K_{r-1})^{s+1} \arrow[rr] \arrow[dd,"\beta"]&&  (K_{r}/K_{r-1})^s& \\
K_{r+1}  \arrow[dr,"\tau"'] &&  \\
	& K_{r} \arrow[rr] && K_{r-1}\arrow[from=uu]
\end{tikzcd}
\]
	\end{deferredproof}
\bibliography{MyBibliography} 
\bibliographystyle{amsplain} 
\end{document}